\documentclass[a4paper,oneside,11pt]{article}

\oddsidemargin=.\paperwidth
\topmargin=.\paperheight

\makeatletter
\newcommand\footnotetext@\relax
\let\footnotetext@\@footnotetext
\newcommand{\MSC}[2][2010]{%
  \unskip\protected@xdef\@thefnmark{}%
  \protect\footnotetext@{\kern-1.8em{\itshape MSC#1\spacefactor3000:}\/ #2}}
\newcommand{\keywords}[1]{
  \unskip\protected@xdef\@thefnmark{}%
  \protect\footnotetext@{\kern-1.8em{\itshape Keywords\spacefactor3000:}\/ #1}}

\newcommand{\email}[1]{%
 \unskip\protected@xdef\@thefnmark{}%
 \protected@xdef\@thanks{\@thanks\protect\footnotetext@{{\itshape Email\spacefactor3000:}\/ \texttt{#1}}}}
\gdef\@date{}
\makeatother

\usepackage[english]{babel}
\usepackage[utf8]{inputenc}
\usepackage{amsmath,amssymb,amsthm}
\usepackage[all]{xy}%
 \entrymodifiers={!!<.em,.6ex>+}
 \SelectTips{cm}{11}
\usepackage{tikz}
\usetikzlibrary{shapes.geometric}
\usepackage[pdftitle={},pdfauthor={}]{hyperref}
\usepackage{enumerate}
\usepackage{comment}

\theoremstyle{definition}
\newtheorem{theorem}{Theorem}[section]
\newtheorem*{thm-intro}{Theorem}
\newtheorem{proposition}[theorem]{Proposition}
\newtheorem{corollary}[theorem]{Corollary}
\newtheorem{definition}[theorem]{Definition}

\newtheorem{example}[theorem]{Example}

\newtheorem{remark}[theorem]{Remark}
\newtheorem{problem}[theorem]{Problem}
\newtheorem{lemma}[theorem]{Lemma}

\newcommand{\from}{\leftarrow}
\newcommand{\Hom}{\operatorname{Hom}}
\newcommand{\id}{\mathrm{id}}
\newcommand{\R}{\mathbb{R}}
\DeclareMathOperator{\rk}{rk}
\newcommand{\xfrom}{\xleftarrow}
\newcommand{\xto}{\xrightarrow}
\newcommand{\Z}{\mathbb{Z}}

\def\id{{\rm id}}

\def\R{{\mathbb R}}
\def\C{{\mathcal C}}

\def\rk{{\rm rk\,}}
\def\xto{\xrightarrow}
\def\xfrom{\xleftarrow}

\def\xmono#1{\overset{#1}{\rightarrowtail}}
\def\toto{\rightrightarrows}
\def\then{\Rightarrow}

\def\action{\curvearrowright}

\def\r#1{|_{#1}}

\def\Hom{{\rm Hom}}

\begin{document}

\title{\bf Higher Vector Bundles}
  \MSC[2020]{Primary 18N50, 58H05\, Secondary 18G35}
\author{Matias del Hoyo%
 \and Giorgio Trentinaglia}
 
\maketitle

\

\begin{abstract} 
We introduce higher analogs for cleavages in the context of (Kan) simplicial fibrations. We apply them to obtain geometric models for representations up to homotopy of (higher) Lie groupoids. Concretely, we set an equivalence between representations up to homotopy and simplicial vector bundles endowed with a cleavage. Our result is an incarnation of the Higher Grothendieck Correspondence, it can be seen as a relative Dold-Kan Theorem and extends earlier work of Gracia-Saz and Mehta on VB-groupoids. 
\end{abstract}

\tableofcontents



\section{Introduction}

Higher Differential Geometry is an active branch of mathematics that uses categorical and simplicial methods to study the topology and geometry of smooth manifolds. (Higher) Lie groupoids play a fundamental role in the theory, providing a unified framework for studying group actions, fiber bundles, and foliations \cite{mm}, they are a natural language for geometric quantization and other quantum theories \cite{dw}, and also a powerful tool for constructing local models and desingularizations in Poisson geometry \cite{cfm}.

An ordinary representation realizes the arrows of a Lie groupoid $G=(G_1\toto G_0)$ as linear isomorphisms between the fibers of a vector bundle $E\to G_0$, respecting identities and composition. Foliated and equivariant bundles are examples of representations. The differentiation of representations yields flat connections on the Lie algebroid of $G$, and Morita equivalent Lie groupoids have equivalent categories of representations. However, these categories may fail to provide complete information about the groupoids themselves: there is neither a general notion of adjoint representation \cite{ac} nor a general Tannaka duality theorem \cite{t1,t2}.

Representations up to homotopy were introduced in \cite{ac}, working with chain complexes of vector bundles, and imposing involved coherence homotopies ruling the compositions. This generalization includes the adjoint and co-adjoint representations as examples, and can easily be adapted to higher Lie groupoids. The 2-term case can be codified by VB-groupoids, studied by Mackenzie \cite{mk} and others. 
The correspondence, set in \cite{gsm}, splits a VB-groupoid by the choice of a cleavage, in a smooth linear version of Grothendieck correspondence between fibered categories and pseudo-functors \cite{sga1}. The VB approach has been useful when studying the Morita invariance and the Lie theory of representations up to homotopy \cite{bcdh,dho}.

A fundamental principle of Higher Category Theory is to work with intrinsic objects, avoiding complicated coherence isomorphisms. As an example, Grothendieck introduced fibered categories, where the existence of pullbacks is emphasized over their explicit construction \cite{sga1}. Another example, proposed by Duskin, involves using the horn-filling to model higher categories, so simplices are composed by forming a horn and picking the remaining face of a filling \cite{duskin}.
This principle brings the problem of finding intrinsic geometric models for representations up to homotopy. In this paper, we present a solution to this problem.

First, given a representation up to homotopy $R: G\action E$, we construct a simplicial vector bundle $q: G\ltimes_RE\to G$, the {\it semi-direct product}. This is a common generalization of both the Dold-Kan construction on a chain complex and the Grothendieck construction on a pseudo-functor.
Then, we define a {\it higher vector bundle} $q: V\to G$ to be a simplicial vector bundle that is a simplicial (Kan) fibration, and define a (higher) {\it cleavage} $C_n\subset V_n$ to be a sub-bundle giving distinguished fillings for the $(n,k)$-horns, $k<n$. Normal weakly flat cleavages are compatible with faces and degeneracies, and weakly flat morphisms preserve some distinguished simplices, see Section \ref{section:hvb} for a precise definition of these notions. Our Main Theorem is the following:


\begin{theorem}\label{thm:main-intro}
Given $G$ a higher Lie groupoid, the semi-direct product yields an equivalence
$$\ltimes:Rep_{\geq0}^\infty(G)\to VB^\infty_+(G)$$
between the category of (nonnegatively graded) representations up to homotopy $R: G\action E$ and the category whose objects are the higher vector bundles $q: V\to G$ endowed with a normal weakly flat cleavage and whose arrows are the weakly flat morphisms.
\end{theorem}

When $G$ is a manifold, this recovers the Dold-Kan correspondence between chain complexes and simplicial vector bundles. When $G$ is a Lie groupoid, this extends the correspondence between ordinary representations and linear action morphisms, see \cite[Prop. 3.3.5]{dh}, and the Gracia Saz-Mehta correspondence between 2-term representations up to homotopy and VB-groupoids \cite{gsm}. 
Even when $G$ is a set-theoretic higher groupoid, our work yields a novel interesting construction.
Lurie's relative nerve $N_fC$ computes the unstraightening or homotopy colimit of a (strict) diagram of simplicial sets $f: C\to sSets$ indexed by a small category \cite[\S 3.2.5]{lurie}, see also \cite{hm}, while our construction $G\ltimes_R E$ models the homotopy colimit of a homotopy diagram $R:G\to Ch_{\geq0}(Vect)$ of chain complexes indexed by a higher groupoid.

We anticipate higher and derived versions for Theorem \ref{thm:main-intro}, and plan to study some of them in future works, keeping things here as simple as possible, aiming to a potentially diverse audience of geometers, topologists, and mathematical physicists. 
In analogy to the 2-term case, we expect our results to help in the problem of integrating representations up to homotopy, see \cite{as}, and to play a role in the Morita invariance of representations up to homotopy, see \cite{ac}. The semi-direct product is used to establish a cohomological version of the Morita invariance in \cite{dhos}.
Further potential applications may include an intrinsic approach to the tensor product of representations up to homotopy \cite{acd}, and explicit formulas for the higher Riemann-Hilbert correspondence \cite{bs}. 

\smallskip

{\bf Organization: }
In Section 2 we set notations and conventions for simplicial sets, and we introduce (higher) cleavages and the notion of flatness. In Section 3 we revisit Dold-Kan and give new formulas 
inspired by the horn-filling property. In Section 4 we review representations up to homotopy of higher Lie groupoids, and Section 5 is where we introduce higher vector bundles and their cleavages. 
Section 6 proves the first half of the theorem, by constructing the semi-direct product. Section 7 proves the other half, by splitting vector bundles with the aid of a suitable cleavage. The final Section 8 discusses future lines of work.

\smallskip

{\bf Acknowledgment: } 
We are grateful to Henrique Bursztyn, Eduardo Dubuc, Tim Porter, Cristian Ortiz, Jim Stasheff, Fernando Studinski, Mahmoud Zeinalian, and Chenchang Zhu for their valuable comments and discussions. We also thank the referees for their insightful feedback. This work was supported by CNPq grant 310289/2020-3 and FAPERJ grant E-26/201.305/2021 (Brazil), and FCT grants SFRH/BPD/81810/2011 and UID/MAT/04459/2020 (Portugal).




\section{Higher groupoids, fibrations, and cleavages}
\label{section:cleavage}

Here we set some notations and conventions about simplicial sets, discuss the horn-filling condition and its role in defining higher categories and fibrations, and introduce (higher) cleavages and the corresponding push-forward operations.

\

The category of finite ordinals $\Delta$ has the ordered sets $[n]=\{n,n-1,\dots,1,0\}$ as objects and the order-preserving maps as arrows.
A {\bf simplicial set} is a contravariant functor $X:\Delta^{\rm op}\to Sets$,  it consists of a set of $n$-simplices $X_n$ for each $n\geq0$, and of structure maps $\theta^*:X_m\to X_n$ for each $\theta:[m]\to[n]$, so $(\theta_2\theta_1)^*=\theta_1^*\theta_2^*$ and $\iota_n^*=\id_{X_n}$, where $\iota_n:[n]\to[n]$ is the identity. A {\bf simplicial map} $f: X\to Y$ is a natural transformation, or equivalently, a collection of maps $f_n: X_n\to Y_n$ such that $f\theta^*=\theta^*f$ for every $\theta$. We write $sSet$ for the category of simplicial sets.


\begin{remark}\label{rmk:simplicial-identities}
Equivalently, 
a simplicial set is a system $(X_n,d_i,u_j)$, where $X_n$ is the set of $n$-simplices for each $n\geq0$, $d_i:X_n\to X_{n-1}$ are the {\bf face} maps and $u_j:X_n\to X_{n+1}$ are the {\bf degeneracy} maps, $0\leq i,j\leq n$, so that the {\bf simplicial identities} hold:
$$d_id_j=d_{j-1}d_i \quad i<j \qquad
d_i u_j=\begin{cases}	u_{j-1}d_i & i<j \\
					1  & i=j,j+1 \\
					u_jd_{i-1} & i>j
					\end{cases}
\qquad u_iu_j=u_{j+1}u_i \quad i\leq j.
$$
The correspondence between both definitions is given by setting $d_i=\delta_i^*$ and $u_j=\upsilon_j^*$, where $\delta_i:[n-1]\to[n]$ is the injection missing $i$ and $\upsilon_j:[n+1]\to[n]$ is the surjection repeating $j$. 
A simplex $x\in X_n$ is {\bf degenerate} if it is in the image of some $u_j$. We warn the reader that this is not the standard notation, where $\sigma_j$ and $s_j$ are used instead of $\upsilon_j$ and $u_j$. Our notation is inspired by the nerve of a groupoid, where the degeneracies include a groupoid unit, and allows us to keep $\sigma_k$ and $s_k$ to denote {\it generalized} sources in Definition \ref{def:flat} and later. 
\end{remark}


The {\bf $n$-simplex} $\Delta[n]\in sSet$ is the the simplicial set represented by $[n]$. By Yoneda, there is a natural bijection $\hom(\Delta[n],X)\cong X_n$, $f\mapsto f(\iota_n)$, so $\Delta[n]$ is freely generated by $\iota_n\in \Delta[n]_n$.
The {\bf $n$-sphere} $\partial\Delta[n]\subset\Delta[n]$ is spanned by the faces $d_i(\iota_n)$, and the {\bf $(n,k)$-horn} $\Lambda^k[n]\subset\Delta[n]$ is spanned by the faces $d_i(\iota_n)$ with $i\neq k$.
Note that $\Lambda^1[1]\cong \Lambda^1[0]\cong\Delta[0]$. 
We denote by $X_{n,k}=\hom(\Lambda^k[n],X)$ the space of $(n,k)$-horns in $X$.

\begin{definition}\label{def:hg}
A simplicial set $X$ is a {\bf higher groupoid} 
if every horn in $X$ has a filling, namely if every {\bf horn map} $d_{n,k}:X_n\to X_{n,k}$ is surjective, and it is of {\bf order $N$}, or $N$-groupoid, if the filling is unique for $n>N$, namely $d_{n,k}$ is bijective for $n>N$. 
\end{definition}

When $n=1$, 
the horn maps $d_{1,0}, d_{1,1}:X_1\to X_0$ identify with the faces $d_1, d_0$.
Intuitively, some simplices $x_1,\dots,x_n\in X_n$ are composable if they are the faces of a horn, a composition $x_0$ is the remaining face of a filling, and any two compositions are homotopic because two fillings fit into a higher horn that can also be filled.

\begin{example} \label{example:nerve} \

\begin{enumerate}[a)]
\item 
Write $\Delta^n$ for the topological $n$-simplex with vertices $e_i$, $0\leq i\leq n$, and $\bar\theta:\Delta^n\to \Delta^m$ for the affine map given by $\bar\theta(e_i)=e_{\theta(i)}$. 
Given $M$ a topological space, its {\bf singular complex} $SM\in sSet$ has continuous maps $x:\Delta^n\to M$ as $n$-simplices and structure maps $\theta^*(x)=x\bar\theta$. The singular complex $SM$ is always a higher groupoid, for $\Delta^n$  retracts onto the union of its faces but one. If $M$ admits a non-constant path then $SM$ is not of order $N$ for any $N$.
\item Regard $[n]$ as a category with an arrow $i\to j$ if $i\leq j$, and each $\theta:[m]\to[n]$ as a functor. Given $C$ a small category, its {\bf nerve} $NC\in sSet$ has functors $g:[n]\to C$ as $n$-simplices and
structure maps $\theta^*(g)=g\theta$. We can visualize $g\in NC_n$ as a chain of composable arrows: 
$$x_n\xfrom{g_n}x_{n-1}\xfrom{g_{n-1}}\dots\xfrom{g_2} x_1\xfrom{g_1}x_0.$$ 
The face $d_i$ erases $x_i$, and either composes $g_{i+1}g_i$, or drops $g_1$ or $g_n$. The degeneracy $u_j$ repeats $x_j$ and inserts an identity. $NC$ is a higher groupoid if and only if $C$ is a groupoid. A higher groupoid has order 1 if and only if it is isomorphic to the nerve of a groupoid.
\end{enumerate}
\end{example}

The relative notion of a higher groupoid is that of a fibration.

\begin{definition} 
A simplicial map $q: X\to Y$ is a {\bf (Kan) fibration} if every horn in $X$ admits a filling with prescribed projection, namely if the {\bf relative horn maps} $d_{n,k}^q: X_n\to X_{n,k}\times_{Y_{n,k}} Y_n$ are surjective. A fibration $q$ is of {\bf order $N$} if $d_{n,k}^q$ is injective when $n>N$. 
$$\xymatrix{\Lambda^k[n] \ar[r]^\forall \ar[d] & X \ar[d]^q \\
 \Delta[n] \ar@{-->}[ur]^{\exists} \ar[r]^\forall & Y}$$
\end{definition}

Note that $X$ is a higher groupoid (of order $N$) if and only if the projection $q: X\to\ast$ is a fibration (of order $N$). If $f:X\to Y$, $g:Y\to Z$ are fibrations of order $N_f,N_g$, then $gf:X\to Z$ is a fibration of order $\max\{N_f,N_g\}$. In particular, if $Y$ is a higher groupoid of order $N_Y$ and $q: X\to Y$ is a fibration of order $N_q$, then $X$ is a higher groupoid of order $\max\{N_Y,N_q\}$. 


In the theory of fibered categories, one can pull back or push forward objects between fibers by using a cleavage, a choice of cartesian arrows playing the role of a connection, see \cite[VI]{sga1}. We now introduce a higher version of cleavage, that is both new and central to our paper, and that allows us to push forward simplices between fibers.

\begin{definition}
Given $q:X\to Y$ a fibration, an {\bf $n$-cleavage} is a subset $C_n \subset X_n$ so that the relative horn map $d_{n,k}^q:C_n\to X_{n,k}\times_{Y_{n,k}} Y_n$ is bijective for all $k<n$. $C_n$ is {\bf normal} if it contains the degenerate simplices.
A {\bf cleavage} $C=\{C_{n}\}$ consists of an $n$-cleavage for each $n$. 
\end{definition}

We can think of a simplex $x\in C_n$ as being {\it cartesian}, {\it thin} or {\it horizontal}. Given $y\in Y_n$ a simplex and $x_{n,k}\in X_{n,k}$ a lifting of the $(n,k)$-horn of $y$, $k<n$, there is a unique lifting $x\in C_n\subset X_n$ of $y$ with $(n,k)$-horn equal to $x_{n,k}$. For $n=1$ this recovers the idea of a distinguished lifting of a path with a prescribed source. If 
the fibration $q$ has order $N$ and $n>N$ then the only $n$-cleavage is $C_n=X_n$, so a cleavage $C$ can be regarded as a finite sequence 
$\{C_1,C_2,\dots,C_N\}$. In particular, if $q$ has order $1$, a cleavage $C$ is the same as a 1-cleavage $C_1$.


\begin{example} \

\begin{enumerate}[a)]
    \item 
    A Serre fibration 
    $q:\tilde M\to M$ yields a simplicial fibration $Sq: S\tilde M\to SM$, and a path-lifting map in the sense of \cite[7.2]{may2} is a 1-cleavage $C_1$ making $d^q_{1,0}|_{C_1}$ a homeomorphism. If $\tilde M$ and $M$ are manifolds, $q$ is locally trivial and $H$ is an Ehresmann connection, we say that a 1-cleavage $C_1$ is {\it compatible} with $H$ if it contains the smooth horizontal paths.
    %
    \item 
    A groupoid morphism $q:\tilde G\to G$ is a fibration if and only if $Nq$ is a simplicial fibration between the nerves, a cleavage $C$ for $Nq$ is the same as a cleavage for $q$, see \cite[VI.7]{sga1}, and is normal if it contains the identities. $C$ induces a {\it fiber pseudo-functor} $F_C:G\dasharrow Gpds$, where $F(x)=q^{-1}(\id_x)$ and $F(g): F(x)\to F(y)$ is a {\em parallel transport} between the fibers.

\end{enumerate}
\end{example}


\begin{remark}\label{rmk:tcp}
Given a cleavage $C$, the last horn map $d^q_{n,n}:C_n\to X_{n,n}\times_{Y_{n,n}}Y_n$ may not be bijective. We need this because in a fibration the fibers are homotopic rather than isomorphic, and a parallel transport is a homotopy equivalence rather than an isomorphim. As a simple example, let $q:X\to Y$ be the groupoid morphism obtained from $\nu_2:[2]\to[1]$ by formally inverting the arrows, so $q(0)=0$ and $q(1)=q(2)=1$. A cleavage must contain $0\from 1$ and $0\from 2$ as liftings of $0\from 1$, and 
$d^q_{1,1}(0\from 1)=(0,0\from 1)=d^q_{1,1}(0\from 2)$, so $d^q_{1,1}$ is not bijective. This is a major difference between our framework and that of twisted cartesian products \cite{may}, where the fibers are isomorphic and the parallel transports are invertible.
\end{remark}

We next introduce another original definition, a flat-like property that will characterize the semi-direct products. Write $\sigma_k:[k]\to[n]$ for the inclusion. Given $X$ a simplicial set, write $s_k=\sigma_k^*:X_n\to X_k$ for the {\it generalized source}, so $s_k(x)=x|_{\sigma_k}\in X_k$ is the first $\dim k$ face of $x$. 

\begin{definition}\label{def:flat}
A cleavage $C$ is {\bf $n$-flat over $S\subset X_0$} if given $x\in C_{n}$ such that $s_0x\in S$,
$s_kx\in C_k$ for all $k>0$, and $d_ix\in C_{n-1}$ for all $i>0$, then $d_{0}x\in C_{n-1}$ as well.
A cleavage $C$ is {\bf flat over $S$} if it is $n$-flat over $S$ for every $n$. When not specified, $S$ is assumed to be $X_0$. 
\end{definition}

\begin{example}\

\begin{enumerate}[a)]
    \item Let $q:\tilde M\to M$ be a locally trivial smooth map and $H$ a complete Ehresmann connection. We say that
    a smooth triangle $x\in S\tilde M_2$ is {\it horizontal} if 
    it maps segments in $\Delta^2$ parallel to $e_0e_1$ into smooth horizontal paths. We say that a cleavage $C$ for $Sq$ is compatible with $H$ if $C_1$ and $C_2$ contain the smooth horizontal paths and triangles. Given $H$, we can always find a compatible cleavage $C$, and if this is 2-flat, then $H$ is flat in the usual smooth sense.  
    \item Let $q:\tilde G\to G$ be a groupoid fibration and $C$ a normal cleavage. 
    Then $C$ is flat if and only if it is closed under composition. A flat cleavage is called a {\em splitting} in \cite[VI.9]{sga1}, and its fiber pseudo-functor is a strict functor. A less restrictive situation is when there is a groupoid section $s: G\to\tilde G$ for $q$. If $C$ is a cleavage containing $s(G_1)$ then it is flat over $s(G_0)$. 
\end{enumerate}
\end{example}

A little simplicial algebra shows that flat cleavages satisfy the following stronger condition.

\begin{lemma}\label{lemma:cleavage-definition}
Let $q: X\to Y$ be a fibration, $C$ a cleavage flat over $S$, and $i_0<n$. If $x\in C_{n}$ is such that $s_0x\in S$,
$s_kx\in C_k$ for all $k>0$, and $d_ix\in C_{n-1}$ for all $i\neq i_0$, then $d_{i_0}x\in C_{n-1}$. 
\end{lemma}

\begin{proof}
Suppose such an $x$ is given. If $i_0=0$ this is by definition, so we can suppose $i_0>0$. We want to show that $z=d_{i_0}(x)\in C_{n-1}$. Let $z'$ be the unique simplex in $C_{n-1}$ with the same $(n-1,0)$-horn and projection as $z$. 
Let $x'$ be the unique simplex in $C_n$ with the same projection as $x$ and with faces $d_i(x')=d_{i}(x)$ for $i\notin\{i_0,0\}$ and $d_{i_0}(x')=z'$.
Note that $d_i(x')\in C_{n-1}$ for all $i\neq0$, and that $s_k(x')=s_{k}(x)\in C_k$ for all $k< n$. Since $C$ is flat over $S$, it follows that $d_0(x')\in C_{n-1}$. Then $d_0(x)=d_0(x')$, for both have the same $(n-1,i_0-1)$-horn, the same projection, and belong to $C$. Analogously, $x=x'$ because both have the same $(n,i_0)$-horn, the same projection, and belong to $C$. Finally, $z=d_{i_0}(x)=d_{i_0}(x')=z'\in C_{n-1}$.  
\end{proof}

A simplicial fibration $q: X\to Y$ has the right lifting property with respect to the anodyne extension, see \cite[I.4]{gj}. 
This is the class of injective simplicial maps $i: K\to L$ containing the horn inclusions $\Lambda^k[n]\to\Delta[n]$ and closed under isomorphisms, co-base changes, retracts, direct sums, and countable compositions. The lifting is not unique in general, but it may be unique if we impose some compatibility with a cleavage. Let us focus on a very specific situation. Given $x\in X_n$, $y=q(x)$, we will define a push-forward $p_i(x)\in X_n$, by transporting $x$ along a lifting of $u_{i+1}(y)$, keeping the position of each vertex $x|_j$, $j\neq i$, and pushing $x|_i$ to the next fiber horizontally over $y|_{\{i,i+1\}}$. We identify an injective  $\alpha:[k]\to[n]$ with its image, and write $x|_\alpha=x\alpha=\alpha^*(x)\in X_k$ for the restriction. 

\begin{lemma}\label{lemma:push-forward}
Given $q:X\to Y$ a simplicial fibration, $C$ a cleavage, $x\in X_n$, $y=q(x)\in Y_n$, and $i<n$, there is a unique $h_i(x)\in C_{n+1}$ such that (i) $d_{i+1}h_i(x)=x$, (ii) $qh_i(x)=u_{i+1}(y)\in Y_{n+1}$, and (iii) $h_i(x)|_\alpha\in C$ for every $\alpha$, $\{i,i+1\}\subset\alpha\subset[n+1]$.
$$\begin{matrix}\xymatrix{
\Delta[n] \ar[r]^x \ar[d]_{\delta_{i+1}} & X \ar[d]^q \\ 
\Delta[n+1] \ar@{-->}[ru]^{h_i(x)} \ar[r]_(.6){u_{i+1}(y)} & Y
}\end{matrix}
\qquad\qquad\qquad\begin{matrix}
\begin{tikzpicture}[scale=.5]
\draw[thick](5,0) -- (1,4);
\draw[thick](3,2) -- (2,2);
\draw[dashed](2,3) -- (2,2);
\draw[dashed](1,4) -- (2,2);
\draw[dashed](4,1) -- (2,2);
\draw[dashed](5,0) -- (2,2);
\draw[fill] (1,4) circle (3pt)node[above right]{\footnotesize$n$};
\draw[fill] (1,4) circle (3pt);
\draw[fill] (2,3) circle (3pt)node[above right]{\footnotesize$i+1$};
\draw[fill] (3,2) circle (3pt)node[above right]{\footnotesize$i$};
\draw[fill] (4,1) circle (3pt);
\draw[fill] (5,0) circle (3pt) node[above right]{\footnotesize$0$};
\draw[fill] (2,2) circle (3pt);
\end{tikzpicture}\end{matrix}
$$
We call $h_i(x)$ the $i$-th {\bf  transport} and $p_i(x)=d_ih_i(x)$ the $i$-th {\bf push-forward} of $x$.  Moreover,
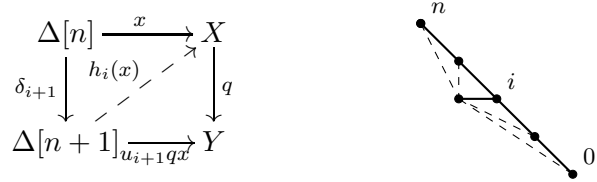
\begin{enumerate}[1)]
    \item If $C$ is normal and $y= u_i(y')$ for some $y'\in Y_{n-1}$, then $h_i(x)=u_i(x)$ and $p_i(x)=x$;
    \item If $C$ is flat over $S$ and $\alpha\subset[n]$ is such that $x|_\alpha\in C$, $x|_{\alpha(0)}\in S$ and $x|_{\alpha\sigma_r}\in C$ whenever $\alpha(r)\leq i$, then $p_i(x)|_\alpha\in C$.
\end{enumerate}
\end{lemma}

\begin{proof}
We start by setting $h_i(x)|_{\{i+1,i\}}$ as the unique edge in $C_1$ over $y|_{\{i+1,i\}}$ with source $x|_i$. We then continue to define $h_i(x)|_\alpha\in C$ for every $\alpha$ such that $\{i+1,i\}\subset\alpha$, inductively on $\dim\alpha$, by iterated applications of the horn-filling.
Fix such an $\alpha$ and let $i=\alpha(i')$. If $j\neq i'$ then $d_j(h_i(x)|_\alpha)$ is already defined by induction, thus we can set $h_i(x)|_\alpha$ as the unique filling in $C$ of the corresponding $i'$-horn over $u_{i+1}(y)|_\alpha$. This proves the existence (and uniqueness) of the 
transport $h_i(x)$ and the push-forward $p_i(x)$.

1) In this case $u_i(x)=x\upsilon_i$ satisfies (i) and (ii). Moreover, $x\upsilon_i|_\alpha$ is degenerate if $\{i,i+1\}\subset\alpha$ and is therefore in $C$. Then $x\upsilon_i=h_i(x)$ is the transport and $p_i(x)=d_i u_i(x)=x$.

2) If $i\notin\alpha$ then $p_i(x)|_\alpha=x|_\alpha\in C$ by hypothesis. 
Suppose then that $i=\alpha(i')$.
Let $\alpha'=\delta_{i+1}\alpha\cup \{i+1\}$. We want to show that $d_{i'}(h_i(x)|_{\alpha'})=p_i(x)|_\alpha\in C$. 
By Lemma \ref{lemma:cleavage-definition}, since $h_i(x)|_{\alpha'(0)}=x|_{\alpha(0)}\in S$, it is enough to check that
$d_{j}(h_i(x)|_{\alpha'})\in C$ for every $j\neq i'$ and $s_r(h_i(x)|_{\alpha'})\in C$ for every $r$. And each of these simplices belongs to $C$, either by hypothesis, or because it agrees with $h_i(x)|_\beta$ with $\{i,i+1\}\subset\beta$.
\end{proof}

\begin{remark}\label{rmk:push-forward-faces}
By construction, 
besides $d_ih_i(x)=p_i(x)$, $d_{i+1}h_i(x)=x$ and $d_ip_i(x)=d_i(x)$, we have the following compatibilities between the faces, the transport, and the push-forward:
$$
h_id_j(x)=\begin{cases}
d_jh_{i+1}(x) & i\geq j\\
d_{j+1}h_i(x) & i+1<j
\end{cases}
\qquad
p_id_j(x)=\begin{cases}
d_jp_{i+1}(x) & i\geq j\\
d_jp_i(x) & i+1<j
\end{cases}$$ 
If $y=u_{i+1}(y')$, then $h_id_{i+1}(x)=d_{i+2}h_i(x)$ and
$p_id_{i+1}(x)=d_{i+1}p_i(x)$, for $d_{i+2}h_i(x)$ is the $i$-th transport of $d_{i+1}(x)$, namely it satisfies (i), (ii) and (iii) of Lemma \ref{lemma:push-forward}.
In general, $d_{i+2}h_i(x)$ and $d_{i+1}p_i(x)$ are like a transport and a push-forward but over a non-degenerate simplex.
\end{remark}


\begin{example}
Let $q:\tilde G\to G$ be a groupoid fibration, $C$ a normal cleavage, 
$g=(x_n\xfrom{g_n}x_{n-1}\xfrom{g_{n-1}}\cdots\xfrom{g_1}x_0)$ in $\tilde G_n$ and $i<n$. Let $c\in C_1$ be such that $s(c)=x_i$ and $q(c)=q(g_{i+1})$.
Then $p_i(g)\in \tilde G_n$ is obtained from $g$ by replacing $g_{i+1}$ by $g_{i+1}c^{-1}$ and $g_i$ by $c g_i$. 
$$\begin{matrix}\xymatrix@R=10pt@C=15pt{
  x_3  &  & & \\
    & x_2 \ar[ul]_{g_3} & & \\
 & x'_1 \ar@{.>}[u]^{g_2c^{-1}} & x_1 \ar[l]_{c}  \ar[lu]_{g_2}   &  \\
 & & &  x_0 \ar[ul]_{g_1} \ar@{.>}[llu]^{cg_1} }\end{matrix} \qquad
p_1(g_3,g_2,g_1)=(g_3,g_2c^{-1},cg_1)$$
\end{example}


A flat connection has no holonomy, the parallel transports over the horizontal and the vertical edges of a small square commute. The following proposition is an incarnation of that principle.

\begin{proposition}\label{prop:commuting-push-forward}
Let $q:X\to Y$ be a simplicial fibration, $C$ 
a cleavage that is normal and flat over $S$, and $i<n$. If
$x\in X_n$ is such that $x|_{k}\in S$ for every $k\leq i$ and $x|_{\beta}\in C$ for every $\beta\subset\sigma_i$, then
$h_{i}h_j(x)=h_{j+1}h_{i}(x)$ for all $i<j$, and
$$
p_ip_j(x)=
\begin{cases}
p_{j}p_i(x) &j>i+1\\
p_{i}p_{i+1}p_{i}(x) & j=i+1\\
p_i(x) & j=i 
\end{cases}$$
\end{proposition}

\begin{proof}
To show $h_{i}h_j(x)=h_{j+1}h_{i}(x)$ we need to prove that $h_{j+1}h_i(x)$ is the $i$-th transport of $h_j(x)$, namely that it satisfies (i), (ii), and (iii) in Lemma \ref{lemma:push-forward}.
By Remark \ref{rmk:push-forward-faces} we have $d_{i+1}h_{j+1}h_i(x)=h_{j}d_{i+1}h_i(x)=h_j(x)$, hence (i). Regarding (ii), $qh_{j+1}h_i(x)=u_{j+2}qh_i(x)=u_{j+2}u_{i+1}(y)=u_{i+1}u_{j+1}(y)$. To prove (iii), we fix $\alpha$ such that $\{i,i+1\}\subset\alpha\subset[n+2]$, and we need to check that $h_{j+1}h_i(x)|_{\alpha}\in C$. If $\{j+1,j+2\}\subset\alpha$ this follows from property (iii) of $h_{j+1}$. If $j+2\notin\alpha$ then $h_{j+1}h_i(x)|_{\alpha}= h_i(x)|_{\upsilon_{j+2}\alpha}\in C$ by property (iii) of $h_i$. 
If $j+2\in \alpha$ and $j+1\notin\alpha$, then $h_{j+1}h_i(x)|_{\alpha}=
p_{j+1}(h_i(x))|_{\upsilon_{j+1}\alpha}$.
We have that $h_i(x))|_{\upsilon_{j+1}\alpha}\in C$ because $\{i,i+1\}\subset\upsilon_{j+1}\alpha$, that $h_i(x))|_{\upsilon_{j+1}\alpha(0)}=x|_{\alpha(0)}\in S$ because $\alpha(0)\leq i$, and that
$h_i(x))|_{\upsilon_{j+1}\alpha\sigma_r}\in C$ for every $r$ because either $\{i,i+1\}\subset \upsilon_{j+1}\alpha\sigma_r$ or $\upsilon_{j+1}\alpha\sigma_r\subset \sigma_i$. Then
$p_{j+1}(h_i(x))|_{\upsilon_{j+1}\alpha}\in C$
by the part 2) of Lemma \ref{lemma:push-forward}. This completes the proof of the identity of the transport simplices.

If $j>i+1$, using Remark \ref{rmk:push-forward-faces} and the above identity, we have $p_ip_j(x)=d_ih_id_jh_j(x)=d_id_{j+1}h_ih_j(x)=d_jd_ih_{j+1}h_i(x)=d_jh_jd_ih_i(x)=p_jp_i(x)$. 


If $j=i+1$, by similar computations, we have
$d_ih_id_{i+1}h_{i+1}d_ih_i(x)=d_id_{i+1}h_{i+1}d_{i+2}h_ih_{i+1}(x)$. Then, writing $x'=h_{i+1}(x)$, we can show that $d_{i+1}h_{i+1}d_{i+2}h_i(x')$ is the $i$-th transport of $d_{i+1}(x')$, by checking (i), (ii) and (iii) in Lemma \ref{lemma:push-forward}, using the same arguments as before. Finally,
$p_ip_{i+1}p_i(x)=d_ih_id_{i+1}h_{i+1}d_ih_i(x)=d_ih_id_{i+1}h_{i+1}(x)=p_{i}p_{i+1}(x)$.

The case $j=i$ follows from part 1) of Lemma \ref{lemma:push-forward}, for $p_i(x)$ sits over $d_iu_{i+1}(y)=u_id_i(y)$.
\end{proof}

\begin{remark}
By iterating push-forward operations, we can send a simplex $x\in X_n$ over $y\in Y_n$ to a vertical simplex in the fiber $q^{-1}(y_n)$. But the push-forward of simplices is not reversible in general, namely $x\mapsto (q(x),p_i(x))$ may not be injective. For instance, in the example in Remark \ref{rmk:tcp}, we have $p_0(0\from 1)=(0\from 0)=p_0(0\from 2)$. Nevertheless, the push-forward is reversible if $C$ satisfies that $d^q_{n,n}:C_n\to X_{n,n}\times_{Y_{n,n}}Y_n$ is also bijective. We call such a cleavage {\bf strict}. 
 Fibrations with a strict cleavage seem to be related to twisted cartesian products \cite{may}.
Our cleavages are not strict in general, so we plan to explore this relation elsewhere.
\end{remark}

\begin{remark}
Our cleavages can be seen as a relative version of Nikolaus's algebraic Kan complexes \cite{nikolaus} and Dakin's $T$-complexes \cite{dakin}. An {\em algebraic Kan complex} is a simplicial set $X$ with choices of horn-fillings $i_{n,k}:X_{n,k} \to X_n$. This is a {\em $T$-complex} if (a) $i_{n,k}(X_{n,k})=C_n$ does not depend on $k$, (b) $C_n$ contains the degeneracies, and (c) if $C_{n-1}$ contains all but one faces of $x\in C_n$, then it contains the other. If $X$ is a $T$-complex, its thin simplices define a normal, flat, and strict cleavage for the projection $X \to \ast$. The converse does not hold, for our flatness axiom involves faces of several dimensions, making it weaker than (c).
\end{remark}


\section{Revisiting the Dold-Kan correspondence}

We review the classical Dold-Kan correspondence, present and discuss a new formula for the inverse functor, and remark the relations with higher groupoids and cleavages.

\

\def\C{{\mathcal C}}

{\bf Simplicial objects} in a category $\C$ are contravariant functors $X:\Delta^{\rm op}\to\C$, they form a category $s\C$ with arrows the natural transformations, and they can be described as graded objects with faces and degeneracy operators exactly as in the case $\C=Sets$. The Dold-Kan correspondence \cite{dold,kan} shows that simplicial objects in an abelian category $\mathcal A$ are essentially the same as chain complexes in nonnegative degrees. We will focus on $\mathcal A=Ab$ the category of abelian groups. The adaptations to the general case are straightforward.

Given $X\in sAb$ a simplicial abelian group, its {\bf normalization} $(NX,\partial)\in Ch_{\geq0}(Ab)$ is the chain complex with $NX_n=\ker (d_{n,0}:X_n\to X_{n,0})=\bigcap_{i>0}\ker(d_i:X_n\to X_{n-1})$ and  $\partial=d_0$.
It turns out that $X_n=NX_n\oplus D_n$, where $D_n$ is the subgroup spanned by the degenerate simplices, that the inclusion yields a quasi-isomorphism  $(NX,d_0)\xto\sim(X,\sum_i(-1)^id_i)$, and an isomorphism $(NX,d_0)\xto\cong (X/D,\sum_i(-1)^id_i)$, see e.g. \cite[\S III.2]{gj}. 




\begin{proposition}[Dold-Kan correspondence]\label{prop:dk-correspondence}
The normalization $N:sAb\to Ch_{\geq 0}(Ab)$ is an equivalence of categories between simplicial abelian groups and chain complexes.
\end{proposition}

The original proof \cite{dold,kan} provides an explicit formula for the inverse, namely
$$N^{-1}:Ch_{\geq 0}(Ab)\to sAb \qquad
N^{-1}(Y)_n=\hom(NF\Delta[n],Y)$$
where $F\Delta[n]\in sAb$ is freely spanned by $\Delta[n]$. This is an instance of the so-called {\it Yoneda extension}: a functor $\varphi:C\to D$ from a small category $C$ to a co-continuous category $D$ produces an adjunction $\phi:Sets^{C^{\rm op}}\rightleftarrows D:\psi$ between  $\phi(X)={\rm colim}_{h^c\to X}\varphi(c)$ and $\psi(d)_c=\hom(\varphi(c),d)$. 

\begin{remark}\label{rmk:labeling}
A chain map $x: NF\Delta[n]\to Y$ can be regarded as a labeling of the $n$-simplex with elements of $Y$. 
Write $\bar\alpha:F\Delta[k]\to F\Delta[n]$ for the morphism induced by $[k]\xto\alpha[n]$, so $\bar\alpha(\iota_k)=\alpha\in\Delta[n]_k$.
Then the abelian group $NF\Delta[n]_k=F\Delta[n]/D_k$ is free and spanned by the injective maps $[k]\xto\alpha[n]$, and the differential $\partial$ is given by $\partial(\alpha)=\sum_{i=0}^k (-1)^i\alpha\delta_i$, where $\delta_i$ are the elementary injections, see Remark \ref{rmk:simplicial-identities}. 
To give a chain map $x:NF\Delta[n]\to Y$ is the same as to label each sub-simplex $[k]\xto\alpha[n]$ of $\Delta[n]$ with an element $\pi_\alpha(x)=x(\alpha)=x\bar\alpha\in Y_k$, in a way such that $\partial(\pi_\alpha(x))=
\sum_{i=0}^k(-1)^i \pi_{\alpha\delta_i}x$.
\end{remark}

%


A more concise inverse to $N$ has been popularized in the literature. We will discuss it in Remark \ref{rmk:dk-literature}.
We now present new formulas that are alternative to those and that catch the geometry of horn-fillings. They play a key role in developing the relative version. The idea is to reconstruct $\Delta[n]$ out of its sub-simplices $\alpha$ containing 0 via a sequence of horn-fillings:
$$
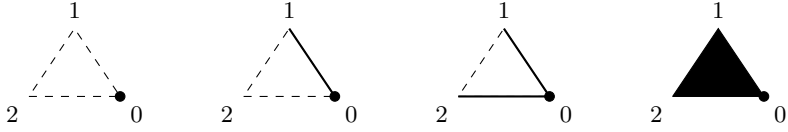
\begin{tikzpicture}[scale=.6]
\draw[dashed](0,0) -- (2,0) node[below right] {\footnotesize $0$}-- (1,1.5) node[above] {\footnotesize $1$} -- (0,0) node[below left] {\footnotesize $2$};
\draw[fill] (2,0) circle (3pt);
\end{tikzpicture}\qquad
\begin{tikzpicture}[scale=.6]
\draw[dashed](0,0) -- (2,0) node[below right] {\footnotesize $0$}-- (1,1.5) node[above] {\footnotesize $1$} -- (0,0) node[below left] {\footnotesize $2$};
\draw[fill] (2,0) circle (3pt);
\draw[thick](2,0) -- (1,1.5);
\end{tikzpicture}\qquad
\begin{tikzpicture}[scale=.6]
\draw[dashed](0,0) -- (2,0) node[below right] {\footnotesize $0$}-- (1,1.5) node[above] {\footnotesize $1$} -- (0,0) node[below left] {\footnotesize $2$};
\draw[fill] (2,0) circle (3pt);
\draw[thick](2,0) -- (1,1.5);
\draw[thick](2,0) -- (0,0);
\end{tikzpicture}\qquad
\begin{tikzpicture}[scale=.6]
\draw[fill](0,0) -- (2,0) node[below right] {\footnotesize $0$}-- (1,1.5) node[above] {\footnotesize $1$} -- (0,0) node[below left] {\footnotesize $2$};
\draw[fill] (2,0) circle (3pt);
\end{tikzpicture}$$

Given $(Y,\partial)\in Ch_{\geq0}(Ab)$, we define the {\bf $n$-simplex abelian group} by $DK(Y)_n$ by
$$DK(Y)_n=\bigoplus_{\substack{[k]\xto\alpha[n]\\ \alpha(0)=0}}Y_{k}$$
where the sum is over the injective maps preserving 0. 
We write $\pi_\beta:DK(Y)_n\to Y_l$ for the projection indexed by $[l]\xto\beta[n]$, and we write $(y,\alpha)$ for the homogeneous element set by $[k]\xto\alpha[n]$ and $y\in Y_k$. Given $[m]\xto\theta[n]$, define $[m+1]\xto{\theta'}[n+1]$ by $\theta'(0)=0$ and $\theta'(i+1)=\theta(i)+1$. 



\begin{proposition}\label{prop:inverse-dk}
The inverse of the normalization $N^{-1}$ is naturally isomorphic to the functor $DK:Ch_{\geq0}(Ab)\to sAb$, where $DK(Y)_n$ is as above, and faces and degeneracies are as follows:
\begin{itemize}
    \item $\pi_\beta u_j=\pi_{\upsilon_j\beta}$, with the convention that $\pi_\gamma=0$ whenever $\gamma$ is not injective;
    \item $\pi_\beta d_i=\pi_{\delta_i\beta}$ if $i>0$; and 
    \item $\pi_\beta d_0=\partial \pi_{\beta'}-\sum_{i=1}^{l+1} (-1)^i\pi_{\beta'\delta_i}$.
\end{itemize}
\end{proposition}


\begin{proof}
For fixed $n$, let us show that the map $N^{-1}(Y)_n\to \bigoplus_{[k]\xto\alpha[n]}Y_k\to DK(Y)_n$ given by composing the inclusion with the projection is an isomorphism.
If $\beta:[l]\to[n]$ does not contain $0$, taking $\alpha=\beta\cup 0$, the description in Remark \ref{rmk:labeling} implies that $\pi_\beta x=\partial \pi_{\upsilon_0\beta'}x-\sum_{i=1}^{l+1}(-1)^i \pi_{\upsilon_0\beta'\delta_i}x$, so the indices containing 0 determine the others, and the map is injective.
Conversely, starting with an arbitrary family $\pi_\alpha(x)\in Y_k$, $[k]\xto\alpha[n]$, $0\in\alpha$, and defining $\pi_\beta x$ with the above equation whenever $0\notin\beta$, it is easy to check that the equation remains true for every $\alpha$, and we get a well-defined chain map $x:NF\Delta[n]\to Y$, proving that the map is surjective.

Regarding the faces, if $i>0$ then $\pi_\beta d_i(x)=x\bar\delta_i\bar\beta=\pi_{\delta_i\beta}x$. 
Regarding the degeneracies, 
$\pi_\beta u_jx=x\bar\upsilon_j\bar\beta=\pi_{\upsilon_j\beta}x$ if $\upsilon_j\beta$ is injective and 0 otherwise, for $x$ must vanish over the degenerate simplices. The formula for $d_0$ follows from
$\partial\pi_{\alpha}=\sum_{i=0}^{k}(-1)^i\pi_{\alpha\delta_i}$ applied to $\alpha=\beta'$.
\end{proof}


\begin{remark}\label{rmk:homog-formulas}
The faces and the degeneracies on the homogeneous elements are given by:
$$
\pi_\beta u_j(y,\alpha)=\begin{cases} 
y & \alpha=\upsilon_j\beta \\
0 & \text{otherwise}
\end{cases}
\qquad
\pi_\beta d_i(y,\alpha)=\begin{cases}
  y & \alpha=\delta_i\beta  \\ 
0 & \text{otherwise}\end{cases}\ (i\neq0)
$$
Regarding $d_0$, there are two cases where the contributions are non-trivial, which we visualize using black dots for $\alpha$ and gray dots for $\delta_0\beta$:
\begin{enumerate}[{Case} I)]
\item If $\alpha=\beta'$  
then $\pi_\beta d_0(y,\alpha)=\partial (y)$:
$$\begin{tikzpicture}
\draw[thick](6,0)node[below] {$0$} --(0,0) node[below] {$n$};
\draw[fill] (6,0.2) circle (1.8pt);
\draw (6.4,0.2) node {\footnotesize $\alpha$};
\draw[fill] (5.5,0.2) circle (1.8pt);
\draw[fill] (4.5,0.2) circle (1.8pt);
\draw[fill] (4,0.2) circle (1.8pt);
\draw[fill] (3,0.2) circle (1.8pt);
\draw[fill] (2,0.2) circle (1.8pt); 
\draw (6,0.45) circle (1.8pt);
\draw (6.4,0.45) node {\footnotesize $\beta'$};
\draw[fill,gray] (5.5,0.45) circle (1.8pt);
\draw[fill,gray] (4.5,0.45) circle (1.8pt);
\draw[fill,gray] (4,0.45) circle (1.8pt);
\draw[fill,gray] (3,0.45) circle (1.8pt);
\draw[fill,gray] (2,0.45) circle (1.8pt);
\draw (5.5,0) node[below] {$1$};
\draw (2,0) node[below] {\footnotesize $\alpha(k)$};
\end{tikzpicture}$$
\item If $\alpha=\beta'\delta_i$ for some $0< i\leq l+1$ 
then $\pi_\beta d_0(y,\alpha)=(-1)^{i+1}y$,
$$\begin{tikzpicture}
\draw[thick](6,0)node[below] {$0$} --(0,0) node[below] {$n$};
\draw[fill] (6,0.2) circle (1.8pt);
\draw (6.4,0.2) node {\footnotesize $\alpha$};
\draw[fill] (5.5,0.2) circle (1.8pt);
\draw[fill] (4.5,0.2) circle (1.8pt);
\draw[fill] (4,0.2) circle (1.8pt);
\draw[fill] (3,0.2) circle (1.8pt);
\draw[fill] (2,0.2) circle (1.8pt); 
\draw (6,0.45) circle (1.8pt);
\draw (6.4,0.45) node {\footnotesize $\beta'$};
\draw[fill,gray] (5.5,0.45) circle (1.8pt);
\draw[fill,gray] (4.5,0.45) circle (1.8pt);
\draw[fill,gray] (4,0.45) circle (1.8pt);
\draw[fill,gray] (3.5,0.45) circle (1.8pt);
\draw[fill,gray] (3,0.45) circle (1.8pt);
\draw[fill,gray] (2,0.45) circle (1.8pt);
\draw (5.5,0) node[below] {$1$};
\draw (3.5,0) node[below] {\footnotesize $\beta'(i)$};
\draw (2,0) node[below] {\footnotesize $\alpha(k)$};
\end{tikzpicture}$$
\end{enumerate}
\end{remark}


\begin{lemma}\label{lemma:ker d_i}
Let $Y$ be a chain complex, and let $x\in 
DK(Y)_n=\bigoplus_{\substack{[k]\xto\alpha[n]\\ \alpha(0)=0}}Y_{k}
$. Then:
\begin{enumerate}[i)]
    \item $x\in\ker d_i$, $i>0$, if and only if $\pi_\alpha x=0$ for all $\alpha$ not containing $i$;
    \item $x\in N(DK(Y))_n=\ker d_{n,0}$ if and only if $\pi_\alpha x=0$ for every $\alpha\neq\iota_n$;
    \item $x\in D_n$, the group spanned by the degenerate simplices, if and only if $\pi_{\iota_n}x=0$.
\end{enumerate}
\end{lemma}

\begin{proof}
i) and ii) are immediate consequences of our formulas for the face maps. Let us prove iii). By definition, $\pi_{\iota_n}u_j=\pi_{\upsilon_j\iota_n}=\pi_{\upsilon_j}=0$, for $\upsilon_j$ is not injective. This proves that $D_n\subset\ker\pi_{\iota_n}$. For the other inclusion, given $\alpha\neq\iota_n$, we must show that $(y,\alpha)\in D_n$. Let $i=\min([n]\setminus\alpha)$, so $\alpha=\delta_i\upsilon_i\alpha$. 
If $i=1$ then $(y,\alpha)=u_0(y,\upsilon_0\alpha)\in D_n$. If $i>1$, we can write $(e,\alpha)=u_{i-1}(e,\upsilon_i\alpha)- (e,\delta_{i-1}\upsilon_i\alpha)$, and since $\min([n]\setminus\delta_{i-1}\upsilon_i\alpha)=i-1$, we conclude by an inductive argument.
\end{proof}

\begin{remark}\label{rmk:dk-literature}
The popular formula for the inverse of the normalization goes as follows. Given $Y$ a chain complex,  $DK'(Y)_n=\bigoplus_{[n]\xto\beta[k]} Y_k$, where the sum runs over the surjective maps. Modulo minor variants, faces and degeneracies are given by
$$
\pi_\beta u_j (e,\alpha)=\begin{cases}
             e & \alpha\upsilon_j=\beta \\
             0 & \text{otherwise}
                      \end{cases}
\qquad
\pi_\beta d_i(e,\alpha)=\begin{cases}
e & \alpha\delta_i=\beta\\
\partial(e) & i=n,\ \alpha\delta_n=\delta_k\beta\\ 
0 & \text{otherwise}.
\end{cases}$$
In \cite[\S 22]{may}, this is presented by writing $\alpha$ as composition of elementary surjections; 
in \cite[\S 1.2.3]{lurie2}, this appears with the variation $(d_i)_{\alpha\beta}=(-1)^n(d_{n-i})_{\alpha\beta}$; in \cite[\S 8.4]{weibel}, a more general formula defines $\theta^*:X_n\to X_m$ for arbitrary $\theta:[m]\to[n]$, as well as in \cite[\S III.2]{gj}. 
There is a duality between epimorphisms $[n]\to[k]$ and monomorphisms $[k]\to[n]$ preserving 0, yielding isomorphisms $DK(Y)_n\cong DK'(Y)_n$, but this is not a simplicial map: in $DK'$ the degeneracies preserve homogeneous elements, and this is not the case in $DK$. 
\end{remark}

\begin{remark}\label{rmk:dk-projective}
Let $R$ be a commutative ring, $Mod^R$ the category of $R$-modules, and $Proj^R_{fg}$ the (non-abelian) subcategory of finitely generated projective ones. Given $X\in sMod^R$ a simplicial $R$-module, the formulas for $DK$ reveal that $X\in sProj^R_{fg}$ if and only if $NX\in Ch_{\geq0}(Proj^R_{fg})$.
Then, the Dold-Kan correspondence restricts to an equivalence $sProj^R_{fg}\cong Ch_{\geq0}(Proj^R_{fg})$. If $M$ is a manifold, by the Serre-Swan Theorem \cite[\S 11]{nestruev}, taking global sections is an equivalence between the category $VB(M)$ of vector bundles over $M$ and $Proj^{C^\infty(M)}_{fg}$.
Combining these results, Dold-Kan yields an equivalence $sVB(M)\cong Ch_{\geq0}(VB(M))$. 
\end{remark}

The next proposition relates simplicial abelian groups with higher groupoids and cleavages.
The first part is standard, see \cite[Lemma I.3.4]{gj}. The second one, since the thin simplices in a $T$-complex are a normal flat cleavage, follows from \cite[\S3, Thm 1.3]{ashley}. We give a proof for completeness.

\begin{proposition}\label{lemma:dk-groupoid}
A simplicial abelian group $X\in sAb$ is a higher groupoid, and it is an $N$-groupoid if and only if $NX_n=0$ for $n>N$.
The degenerate simplices span the unique normal cleavage $D_n\subset X_n$ for the projection $X\to\ast$, and $D$ is flat.
\end{proposition}

\begin{proof}
By Dold-Kan, the problem of extending a horn $x:\Lambda^{k}[n]\to X$ to the whole simplex 
is equivalent to extend 
$x^\#:NF\Lambda^{k}[n] \to NX$ to $NF\Delta[n]$. 
But $NF\Delta[n]=NF(\Lambda^{k}[n])\oplus F(\iota_n,\partial \iota_n)$ as chain complexes.  
Thus, to extend $x^\#$ is the same as to give a chain map $F(\iota_n,\partial\iota_n)\to Y$. Such a map is determined by picking $\phi(\iota_n)\in Y_n$, and the choice is unique if and only if $Y_n=NX_n=0$.

Write $Y'$ for the $(n-1)$-truncation of $Y$, so $Y'_k=Y_k$ if $k<n$ and $Y'_k=0$ otherwise. It follows from Lemma \ref{lemma:ker d_i}
that $D_n\to DK(Y)_{n,k}$ identifies with $DK(Y')_{n}\to DK(Y')_{n,k}$, and since $DK(Y')$ is a higher groupoid, this is an isomorphism. This proves that $D$ is indeed a normal cleavage. To show that $D$ is flat, let $w\in D_{n+1}$ such that $s_kw\in D_k$ for all $k>0$ and $d_iw\in C_n$ for all $i>0$. Then
$\pi_{\sigma_k}w=0$ for all $k>0$ and $\pi_{\delta_i}w=0$ for all $i>0$. 
It follows that 
$\pi_{\iota_n}d_0w=\partial \pi_{\iota_{n+1}}w - \sum_{0<i\leq l+1}(-1)^i\pi_{\delta_i} w = 0$
and therefore $d_0w\in D_n$.
\end{proof}





\section{Higher Lie groupoids and their representations}
\label{sec:representations}

We give here a quick overview of Lie groupoids and higher Lie groupoids, and of representations and representations up to homotopy. We pay special attention to the 2-term case and its correspondence with VB-groupoids, which was our motivation for the Main Theorem \ref{thm:main-intro}.

\


A {\bf Lie groupoid} $G=(G_1\toto G_0)$ is a groupoid where both objects and arrows form smooth manifolds, the source and the target are submersions $s,t: G_1\to G_0$, and the multiplication, unit and inverse are smooth maps $m: G_1\times_{G_0}G_1\to G_1$, 
$u: G_0\to G_1$, 
$i: G_1\to G_1$. 
We regard $G_0$ as an embedded sub-manifold of $G_1$ via $u$. We refer to \cite{mm,mk,dh,bdh} for more details, constructions, and examples.


\def\U{{\mathcal U}}

\begin{example}\

\begin{enumerate}[a)]
    \item If $M$ is a manifold then its {\bf unit groupoid} $M\toto M$ has only identity arrows. If $K$ is a {\bf Lie group} then it can be regarded as a Lie groupoid with only one object $K\toto\ast$. Generalizing these two examples, a Lie group action $\rho: K\action M$ gives rise to a {\bf translation groupoid} $Tr(R)=(K\times M\toto M)$ where $s$ is the projection and $t$ is the action.  
    \item A submersion $q: M\to N$ yields a {\bf submersion groupoid} $M\times_N M\toto M$ with one arrow between two points if they are in the same fiber. When $N=\ast$ this is the {\bf pair groupoid} $M\times M\toto M$. The {\bf Cech groupoid} $\coprod U_{j}\cap U_i\toto\coprod U_i$ of an open cover $\{U_i\}_{i\in I}$ is the submersion groupoid of $\coprod U_i\to M$. The {\bf holonomy groupoid} $Hol(F) \rightrightarrows M$ of a foliation $F \subset TM$ is the colimit of the submersion groupoids of the foliated charts, modulo paths with trivial holonomy.
    %
\end{enumerate}    
\end{example}


The nerve $(G_n,d_i,u_j)$ of a Lie groupoid $G$ is a simplicial manifold, for
$G_n$ is a fibered product $G_{n-1}\times_{G_0} G_1$ along a submersion, and it satisfies a smooth version of Definition \ref{def:hg}. 

\begin{definition}
A simplicial manifold $G$ is a {\bf higher Lie groupoid} if for every $k,n$ the horn space $G_{n,k}\subset\prod_{i\neq k}G_{n-1}$ is an embedded submanifold and
the horn map $d_{n,k}: G_n\to G_{n,k}$ is a surjective submersion.
\end{definition}

There is some redundancy in this definition. 
If the horn maps up to dimension $n-1$ are surjective submersions, then $G_{n,k}$ is smooth. We refer to \cite{h,zhu,bg,dhos} for more details, examples, and further discussions on higher Lie groupoids.

\begin{remark}\label{rmk:1-groupoid}
One can check that the nerve functor gives an equivalence between Lie groupoids and higher Lie groupoids of order 1 \cite[Prop 2.4]{wolfson}. 
From now on, we will identify a Lie groupoid with its nerve. Our examples and motivations come from the theory of Lie groupoids, but most of our results are valid for higher Lie groupoids, so we will work in this generality.
\end{remark}


Let $G$ be a higher Lie groupoid.
Write $\sigma_k,\tau_k:[k]\to[n]$ for the inclusions given by $\sigma_k(i)=i$ and $\tau_k(i)=i+n-k$, and write $s_k=\sigma_k^*$, $t_k=\tau_k^*$ for the induced maps $G_n\to G_k$. Thus, if $g\in G_n$ is an $n$-simplex, then $s_k(g)$ is the front dimension $k$ face and $t_k(g)$ is the back dimension $k$ face.
The smooth functions over $G$ define a differential graded algebra $C(G)=\bigoplus_p C^\infty(G_p)$, with differential $\partial=\sum_{i}(-1)^id_i^*: C^\infty(G_p)\to C^\infty(G_{p+1}) $, 
and cup product
$$(f_1\cup f_2)(g)=f_1(t_qg)f_2(s_pg)\qquad g\in G_{p+q},\ f_1\in C^q(G),\ f_2\in C^p(G)$$


Given $G$ a higher Lie groupoid and $E\to G_0$ a vector bundle,
a {\bf representation} $R:G\action E$ is a section $R\in\hom (G_1,\hom(s^*E,t^*E))$ associating to each arrow $y\xfrom g x$ in $G$ a linear isomorphism $R^g:E^x\to E^y$ so that $R^{u(x)}=\id_{E^x}$ for every $x\in G_0$ and $R^{d_1g}=R^{d_0g}R^{d_2g}$ for every $g\in G_2$. 
If $G$ is a Lie groupoid and $g=(k,h)\in G_2$ is a pair of composable arrows, the last equation reads $R^{kh}=R^kR^h$.

\begin{example}\

\begin{enumerate}[a)]
    \item Representations of a unit groupoid are just vector bundles. Representations of a Lie group viewed as a Lie groupoid are the usual ones. Representations of a translation groupoid $Tr(R)$ arising from a group action $K\action M$ are the same as $K$-equivariant vector bundles over $M$.
    \item A representation of a submersion groupoid is a {\it descent datum}, see \cite{sga1}. In particular, a representation of a pair groupoid is a trivialization, and a representation of a Cech groupoid on a trivial vector bundle is a $GL_n$-cocycle. A representation of a holonomy groupoid $Hol(F)\toto M$ is the same as a foliated vector bundle.
\end{enumerate}    
\end{example}


Next, we present alternative approaches to representations. Given $G$ a higher Lie groupoid and $E\to G_0$ a vector bundle, an {\bf $E$-valued cochain} is a section 
$c\in C^p(G,E)=\Gamma(G_p,t_0^*E)$, and it is {\bf normalized} if it vanishes over the degenerate simplices. The space $C(G,E)^\bullet=\bigoplus_p C^p(G,E)$ has a canonical $C(G)$-module structure, given by 
$$(c\cdot f)(g)=c(t_qg)f(s_pg)\qquad g\in G_{p+q},\ c\in C^q(G,E),\ f\in C^p(G).$$
A higher Lie groupoid morphism $q:X\to Y$ is a {\bf fibration} if for every $k,n$ the relative horn space $X_{n,k}\times_{Y_{n,k}} Y_n\to \prod_{i\neq k}X_{n-1}\times Y_n$ is an embedded submanifold and the relative horn map $d^q_{n,k}: X_n\to X_{n,k}\times_{Y_{n,k}} Y_n$ is a surjective submersion. We say that $q$ has order $N$ if $d^q_{n,k}$ is injective for $n>N$. 


\begin{proposition}\label{prop:representation} 
A representation $R: G\action E$ is equivalent to any of the following:
\begin{enumerate}[i)]
    \item A degree 1 differential $D$ in $C(G;E)$ preserving the normalized cochains and satisfying the Leibniz rule $D(c\cdot f)=D(c)\cdot f+(-1)^{q}c\cdot\partial f$, where $c\in C^q(G,E)$ and $f\in C^p(G)$;
    \item A simplicial vector bundle $q:X\to G$ that is a fibration of order 0 and such that $X_0=E$.
\end{enumerate}
\end{proposition}

The complex $(C(G,E),D)$ are the cochains of $G$ with values in $E$, and its {\bf cohomology} is $H(G,E)$. 
The (higher) Lie groupoid $X$ appearing in ii) is often called {\bf translation groupoid} of $R$ and denoted $Tr(R)$.

\begin{proof}
The first equivalence is a particular case of Proposition \ref{prop:ruth}. The relation between $R$ and $D$ is given by the equation 
$$R^gc(x)=D(c)(g)+c(y) \qquad y\xfrom g x\in G_1, \ c\in C^0(G,E)=\Gamma(E).$$
Given $R$, one can define $X_n=G_n\times_{G_0}E$, where the pullback is through $s_0:G_n\to G_0$, and if $\theta:[m]\to[n]$, then $\theta^*:X_n\to X_m$ is given by $\theta^*(g,v)=(\theta^*(g),R^{g|_{\theta(0)\from 0}}(v))$. Conversely, if $q:X\to G$ is a simplicial vector bundle and an order 0 fibration then the horn maps $d_{n,k}:X_n\to X_{n,k}$ are fiberwise isomorphisms, and we can define
$$R^g(e)= t|_{X^g}\circ s|_{X^g}^{-1}(e,g))\qquad e\in E^x,\ y\xfrom g x\in G_1.$$
These constructions are mutually inverse. For more details when $G$ is a Lie groupoid see \cite[Section 3.3]{dh},\cite[Thm 2.7]{gsm}. The extensions to higher Lie groupoids are straightforward.
\end{proof}




Representing higher Lie groupoids on vector bundles is not quite satisfactory, as discussed in the introduction. 
In \cite{ac} a convenient generalization was introduced, working with (bounded) graded vector bundles.  
Given $E=\bigoplus_n E_n\to M$ a graded vector bundle and $x\in M$, write $E^x=\bigoplus_nE_n^x$ for the fiber over $x$. Given $E,E'\to M$, write  $\Hom^{(k)}(E,E')\to M$ for the graded vector bundle of degree $k$ morphisms, so $\Hom^{(k)}(E,E')^x_n=\{\phi_n:E_n^x\to {E'}_{n+k}^x\}$.

\begin{definition}
Given $G$ a higher Lie groupoid and $E=\bigoplus_{n}E_n$ a graded vector bundle over $G_0$, a {\bf representation up to homotopy} $R:G\action E$ consists of a sequence $(R_m)_{m\geq 0}$
$$R_m\in\Gamma\big(G_m;\Hom^{(m-1)}(s_0^*E,t_0^*E)\big) \qquad R_m^g:E_\bullet^{s_0(g)}\to E_{\bullet+m-1}^{t_0(g)}$$
satisfying the following two conditions:
\begin{enumerate}[RH1)]
    \item $R_1^{u(x)}=\id$ for $x\in G_0$, and $R_m^g=0$ for $g\in G_m$ degenerate, $m>1$; and
    \item $\sum_{k=1}^{m-1}(-1)^k R_{m-1}^{d_kg}=\sum_{r=0}^m(-1)^r R_{m-r}^{t_{m-r}g}\circ 
R_r^{s_{r}g}$ for $g\in G_m$, $m\geq 0$.
\end{enumerate}
We say that $R$ has {\bf order} $N$ if $E_n=0$ for all $n>N$, and $R$ is {\bf strict} if $R_m=0$ for all $m>1$.
\end{definition}

The first operators $R_m$ admit simple interpretations: $R^x_0$ is a chain differential on $E^x_\bullet$ for each $x\in G_0$; $R^g_1:(E^{x_0}_\bullet,R^{x_0}_0)\to(E^{x_1}_\bullet,R^{x_1}_0)$ is a chain map for each $x_1\xfrom g x_0\in G_1$; and $R^{g}_2:R^{d_1g}_1\then R^{d_0g}_1R^{d_2g}_1$ is a chain homotopy for each $g\in G_2$. The higher operators $R_m$ give higher homotopies among the previous $R_k$, $k<m$. 


\begin{definition}
Given $R:G\action E$ and $R':G\action E'$ two representations up to homotopy, a {\bf morphism}
$\psi:R\to R'$ consists of a sequence $(\psi_m)_{m\geq0}$
$$\psi_m\in\Gamma\big(G_m;\Hom^{(m)}(s_0^*E,t_0^*E)\big) \qquad \psi_m^g:E_\bullet^{s_0(g)}\to E_{\bullet+m}^{t_0(g)}$$
satisfying the following two conditions:
\begin{enumerate}
    \item[RH3)] $\psi_m^g=0$ for $g\in G_m$ degenerate, $m>0$; and  
    \item[RH4)] $\sum_{r=0}^m (-1)^{m}{R'}^{t_{m-r}g}_{m-r} \psi^{s_rg}_r+ \sum_{i=1}^{m-1} (-1)^{i} \psi^{d_ig}_{m-1}=
\sum_{r=0}^m(-1)^{r}\psi_{m-r}^{t_{m-r}g}{{R}}^{s_rg}_{r}$ for $g\in G_m$, $m\geq0$.
\end{enumerate}
\end{definition}

The first degree of a morphism $\psi_0:(E,R_0)\to (E',R'_0)$ is a chain map, and the higher $\psi_r$ are homotopies among the various compositions. 
When $\psi_m=0$ for every $m>0$ we say that $\psi$ is a {\bf strict morphism}. When $E=E'$ and $\psi_0=\id$ we say that $\psi$ is a {\bf gauge equivalence}.
We write $Rep^\infty(G)$ for the category of representations up to homotopy, and $Rep_{\geq0}^\infty(G)$, $Rep^N(G)$, $Rep_{\geq0}^N(G)$ for the full subcategories of objects of nonnegative degrees and/or order $N$.


\begin{example}\ \label{ex:ruth}
Let $G$ be a Lie groupoid.
 
\begin{enumerate}[a)]
    \item When $G=M$ is just a manifold, viewed as a unit groupoid, then a representation up to homotopy $M\action E$ is the same as a chain complex of vector bundles. 
    \item When $N=0$, so $E=E_0$ is just a vector bundle, then a representation up to homotopy $R: G\action E$ is the same as a usual representation, as defined before. 
    \item When $N=1$, so $E=E_0\oplus E_1$, 
    a representation up to homotopy $R:G\action E$ gives for each $x\in G_0$ a differential $R^x_0:E^x_1\to E^x_0$, for each $x_1\xfrom g x_0\in G_1$ a chain map $R^g_1:(E^{x_0},R^{x_0}_0)\to(E^{x_1},R^{x_1}_0)$ and for each $(g_2,g_1)\in G_2$ a chain homotopy $R^{h,g}_2:R^{hg}_1\then R^h_1R^g_1$ satisfying
$$-R_2^{g_3,g_2g_1}+R_2^{g_3g_2,g_1}=-R_2^{g_3,g_2}R_1^{g_1}+R_1^{g_3}R_2^{g_2,g_1}\qquad (g_3,g_2,g_1)\in G_3.$$
In \cite{gsm} they use the notations $R_0=\partial$, $R_1=\Delta$ and $R_2=-\Omega$ with a misleading sign.
\end{enumerate}    
\end{example}


If $R: G \action E$ is a representation up to homotopy and $y \xfrom{g} x \in G_1$, then $R^g: (E^x, R_0^x) \to (E^y, R_0^y)$ is a quasi-isomorphism. In fact, if $h \in G_2$ is such that $d_2 h = g$ and $d_1 h = \id_x$, then $R^{d_0 h}: E^y \to E^x$ is a quasi-inverse. But $R^g$ is not invertible in general, see for instance the adjoint representation of the pair groupoid $S^2 \times S^2 \toto S^2$ \cite[Rmk 5.6]{dhs}. 

\begin{lemma}\label{lemma:cycles-borders}
Given $R:G\action E$ a representation up to homotopy and $y\xfrom g x\in G_1$,
then $(E^x,R^x_0)$ and $(E^y,R^y_0)$ are isomorphic chain complexes (even though $R^g$ may not be an isomorphism).
\end{lemma}

\begin{proof}
Denote by $Z^x_n$, $B^x_n$ $H^x_n$ the point-wise cycles, boundaries, and homology. 
We know that $\dim E^x_n=\dim E^y_n$ because $E_n$ is a vector bundle and that $\dim H^x_n=\dim H^y_n$ for $R^g$ is a quasi-isomorphism. Consider the following natural short exact sequences:
$$0\to Z^x_n\to E_n^x\to B_{n-1}^x\to 0
\qquad
0\to B^x_n\to Z^x_n\to H^x_n\to 0$$
If $\dim B^x_{n-1}=\dim B^y_{n-1}$ then $\dim Z^x_n=\dim Z^y_n$ by the first sequence, and then $\dim B^x_n=\dim B^y_n$ by the second sequence.
Since our graded vector bundles are bounded, $B^x_{n}=0= B^y_{n}$ for $n$ small enough. By an inductive argument, 
$\dim B^x_n=\dim B^y_n$ for every $n$. The claim follows by using non-natural splittings $E_n^x\cong H_{n}^x\oplus B_{n}^x\oplus B_{n-1}^x$ under which $R_0^x(a,b,c)=(0,c,0)$.
\end{proof}

It becomes natural to explore alternative approaches to representations up to homotopy, in analogy to Proposition \ref{prop:representation}, and this comprises two problems: the equivalence with the differential graded algebra approach, and the equivalence with the fibration approach.
The first problem was solved in \cite[Prop. 3.2]{ac} for Lie groupoids, and the extension for higher Lie groupoids is straightforward. Given $G$ a higher Lie groupoid and $E= \bigoplus_{n} E_n$ a graded vector bundle over $G_0$, an {\bf $E$-valued cochain} is an element 
$$c\in C^{n}(G,E) =\bigoplus_{i+j=n}C^i(G, E_{-j})$$
where $C^m(G, E_{n})=\Gamma(G_m,t_0^*E_{n})$, and it is {\bf normalized} if it vanishes over degenerate simplices.
The space $C(G, E)=\bigoplus_{n}C^n(G, E)$ has a canonical $C(G)$-module structure, given by 
$$(c\cdot f) (g) = c(t_qg)f(s_pg)\qquad g\in G_{p+q},\ c\in C^q(G, E),\ f\in C^p(G).$$

\begin{proposition}\label{prop:ruth}
A representation up to homotopy $R:G\action E$ is the same as a degree 1 differential $D$ in $C(G;E)$ preserving the normalized cochains and satisfying the Leibniz rule $D(c\cdot f)=(Dc)\cdot f+(-1)^{q}c\cdot(\partial f)$. A morphism is the same as a $C(G)$-module morphism commuting with the differentials.
\end{proposition}

\begin{proof}
The Lie groupoid case was proven in \cite[Prop. 3.2]{ac}. Explicit formulas with our conventions are given in \cite[Lema 3.2]{dhos}. The higher Lie groupoid case is analogous.
\end{proof}




The second problem, of characterizing representations up to homotopy from the fibration viewpoint, is the main subject of this paper, and a solution is presented in Theorem \ref{thm:main-intro}. 
The 2-term case was previously solved in \cite{gsm}, one of our inspirations, where 2-term representations up to homotopy are compared with VB-groupoids. A {\bf VB-groupoid} $q:\Gamma\to G$ is a Lie groupoid morphism such that $\Gamma_0\to G_0$ and $\Gamma_1\to G_1$ are vector bundles, and the structure maps of $\Gamma$ are linear \cite{mk}. 
The paradigmatic examples are the {\bf tangent} and {\bf cotangent} VB-groupoids $(TG\toto TM)\to (G\toto M)$ and $(T^*G\toto A^*)\to (G\toto M)$. Equivalently, a VB-groupoid is the same as a simplicial vector bundle $\Gamma\to G$ over a Lie groupoid that is an order 0 fibration.


\begin{remark}\label{rmk:grothendieck}
Given $G=(G_1\toto G_0)$ a Lie groupoid and $R:G\action E_1\oplus E_0$ a 2-term representation up to homotopy, its {\bf Grothendieck construction} $Gr(R)=(t^*E_1\oplus s^*E_0\toto E_0)$ 
is a VB-groupoid over $G$, with source, target, and multiplication given by:
$$s(c,g,e)=e \quad t(c,g,e)=R^y_0c+R^g_1e \quad 
(c',g',e')(c,g,e)=(c'+R_1^{g'}c+R^{g',g}_2e,g'g,e).$$ 
If $E_1=0$ then $R$ is a representation and $Gr(R)=Tr(R)$ is the translation groupoid.
Given $q:\Gamma\to G$ a VB-groupoid, we can split it into a representation up to homotopy $R: G\action (\ker s|_{G_0}\oplus\Gamma_0)$ by using a cleavage. 
The Grothendieck construction is an equivalence of categories between $Rep^2_{\geq0}(G)$ and the category of VB-groupoids over $G$, see \cite{gsm,dho}. 
\end{remark}


\begin{remark}
Representations up to homotopy make sense for arbitrary simplicial sets $X$ \cite[Definition 2.2]{as}. Our examples and motivations come from Lie groupoids, but our techniques and results easily adapt to the higher Lie groupoid case, thus this is the context in which we have formulated the final version of our work. This includes the set-theoretic higher groupoid case. When $X$ is merely a simplicial set, representations up to homotopy are far more intricate, akin to the shift from group to monoids. We plan to explore this case in future work. 
\end{remark}


\section{Introducing higher vector bundles}
\label{section:hvb}


We introduce higher vector bundles and their cleavages, offer alternative formulations, and present the main examples. We also discuss the morphisms between higher vector bundles and define the categories $VB(G)$ and $VB^+(G)$.

\

Let $G$ be a higher Lie groupoid, and let $G_{n,k}$ be the space of $(n,k)$-horns, which is a closed embedded submanifold of $\prod_{i\neq k}G_{n-1}$. 

\begin{definition}
We define a {\bf higher vector bundle} $q: V\to G$ to be a simplicial vector bundle that is also a (set-theoretic) simplicial fibration. It is an {\bf $N$-vector bundle} if it has order $N$. A morphism of higher vector bundles $\psi:V\to V'$ is a smooth simplicial map that is linear over $G$.
\end{definition}

We restrict our attention to higher Le groupoids and higher vector bundles of finite order.
We write $VB^\infty(G)$ for the category of higher vector bundles over $G$, and $VB^N(G)$ for the full subcategory of objects of order $N$. The total space $V$ is a higher Lie groupoid, see Proposition \ref{prop:equiv-def}. Note that if $N_G$, $N_V$ and $N_q$ are the orders of $G$, $V$ and $q$, then $N_V=\max\{N_G,N_q\}$. 


\begin{example}\

\begin{enumerate}[a)]
    \item A higher vector bundle over a manifold is the same as a simplicial vector bundle $q:V\to M$. In fact, since $M$ is discrete as a simplicial set, $q$ is a fibration if and only if $V$ is a higher groupoid, and this is always the case, see Remark \ref{rmk:dk-projective} and Lemma \ref{lemma:dk-groupoid}. Moreover, $q: V\to M$ is an $N$-groupoid if and only if its normalization $E=N(V)$ satisfies $E_n=0$ for all $n>N$.
    \item A 0-vector bundle over a Lie groupoid is the same as a Lie groupoid representation. Given $R: G\action E$, the projection $q:Tr(R)\to G$ is a 0-vector bundle, and conversely, if $q: V\to G$ is a 0-vector bundle, then $V$ is a Lie 1-groupoid, see Remark \ref{rmk:1-groupoid} and Proposition \ref{prop:representation}. 
    \item A 1-vector bundle over a Lie groupoid is the same as a VB-groupoid.  If $q:\Gamma\to G$ is a VB-groupoid, then $s,t:\Gamma_1\to\Gamma_0$ are surjective submersion, hence fiberwise surjective, and $q$ is a fibration. And if $q:V\to G$ is a 1-vector bundle, $V$ is a Lie groupoid, 
    see Remark \ref{rmk:1-groupoid}.
    \item Given $G$ a higher Lie groupoid, its tangent $TG\to G$ is a higher vector bundle. In fact, there is a canonical identification $(TG)_{n,k}\cong T(G_{n,k})$, and since $d_{n,k}:G_n\to G_{n,k}$ is a surjective submersion, the relative horn map $d^\pi_{n,k}:TG_n\to TG_{n,k}\times_{G_{n,k}}G_n$ also is. This example encodes the adjoint representation of a higher Lie groupoid, and is treated in \cite{ttt}.
\end{enumerate}
\end{example}


In a higher vector bundle, the horn-space $q: V_{n,k}\to G_{n,k}$ is naturally an embedded sub-bundle of $\prod_{i\neq k}V_{n-1}\to\prod_{i\neq k}G_{n-1}$, via the map $\prod_{i\neq k}d_i$. 
This can be seen as an instance of the next general lemma, see also \cite[2.10]{wolfson}. 

\begin{lemma}\label{lemma:smooth-horns}
Let $q:V\to G$ be a higher vector bundle and $S\subset\Delta[n]$ be a collapsible sub-simplicial set. Then $\hom(S,V)\to\hom(S,G)$ is a vector subbundle of $\prod_{\substack{[k]\xmono\alpha[n]\\ \alpha\subset S}}V_{k}\to \prod_{\substack{[k]\xmono\alpha[n]\\ \alpha\subset S}}G_{k}$.
\end{lemma}

\begin{proof}
Write $\hom(S,X)\cong X_{S}$. 
Since $S$ is collapsible, it admits a filtration $S=S_a\supset S_{a-1}\supset\dots\supset S_1\supset S_0=\ast$ and isomorphisms $S_k\cong S_{k-1}\cup_{\Lambda^{j_k}[n_k]}\Delta[n_k]$.
We argue inductively on $m=\min(a,\dim S)$. If $m=0$ the result is clear. In the general case, we have set-theoretic fibered product diagrams in the total spaces and in the bases:
$$\xymatrix{
V_{S_a} \ar[r] \ar[d] & V_{S_{a-1}} \ar[d]\\
V_{n_a} \ar[r] & V_{n_a,j_a}
}
\qquad
\xymatrix{
G_{S_a} \ar[r] \ar[d] & G_{S_{a-1}} \ar[d]\\
G_{n_a} \ar[r] & G_{n_a,j_a}
}$$
By inductive hypothesis we have that $V_{n_a,j_a}\to G_{n_a,j_a}$ and $V_{S_{a-1}}\to G_{S_{a-1}}$ are well-defined  sub-bundles of $\prod_{i\neq j_a}V_{n_a-1}\to \prod_{i\neq j_a}G_{n_a-1}$ and $\prod_{\substack{[k]\xmono\alpha[n]\\ \alpha\subset S_{a-1}}}V_{k}\to \prod_{\substack{[k]\xmono\alpha[n]\\ \alpha\subset S_{a-1}}}G_{k}$. Moreover, since $V\to G$ is a fibration and $G$ is a higher Lie groupoid, $V_{n_a}\to V_{n_a,j_a}$ is fiberwise surjective and covers a surjective submersion. It follows from the standard transversality criterion that $V_S\to G_S$ is a vector sub-bundle of $\prod_{\substack{[k]\xmono\alpha[n]\\ \alpha\subset S}}V_{k}\to \prod_{\substack{[k]\xmono\alpha[n]\\ \alpha\subset S}}G_{k}$, see e.g. \cite[A.2.3]{bcdh}. 
\end{proof}



The following proposition provides some alternative ways to define higher vector bundles.

\begin{proposition}\label{prop:equiv-def}
Let $G$ be a higher Lie groupoid and $q: V\to G$ be a simplicial vector bundle. The following are equivalent:
\begin{enumerate}[i)]
    \item $q$ is a set-theoretic simplicial fibration; 
    \item $q$ is a smooth simplicial fibration;
    \item $V$ is a higher Lie groupoid.
\end{enumerate}
\end{proposition}



\begin{proof} 
Consider the following diagram:
    $$\xymatrix{V_n \ar@/^/[rdr]^{d_{n,k}} \ar@/_/[drd]_{q_n} \ar[rd]^(.6){d^q_{n,k}}& & \\ 
    & G_n\times_{G_{n,k}} V_{n,k} \ar[d] \ar[r] & V_{n,k} \ar[d]^{q_{n,k}}  \\ 
    & G_n \ar[r]_{d_{n,k}} & G_{n,k} }$$

\begin{itemize}
    \item {\bf i) $\then$ ii)} By Lemma \ref{lemma:smooth-horns} we have that $V_{n,k}\to G_{n,k}$ is a vector bundle. 
Since $d^q_{n,k}$ is surjective and also a vector bundle map covering $\id_{G_n}$, it must be fiber-wise surjective. Then $d^q_{n,k}$ is a vector bundle epimorphism and, in particular, a surjective submersion, proving ii).
    \item {\bf ii) $\then$ iii)} We know that $d^q_{n,k}:V_n\to G_n\times_{G_{n,k}}V_{n,k}$ is a surjective submersion and a vector bundle map, then it is fiberwise surjective, and so does $d_{n,k}:V_n\to V_{n,k}$. Since $d_{n,k}:G_n\to G_{n,k}$ is a surjective submersion,  $d_{n,k}:V_n\to V_{n,k}$ also is, proving iii).
\item {\bf iii) $\then$ i)} $V_{n,k}\to G_{n,k}$ is a vector bundle, for it is an embedded submanifold of $\prod_{i\neq k} V_{n-1}$ invariant under the multiplication by scalars \cite[Thm 2.3]{gr}. Since $d_{n,k}:V_n\to V_{n,k}$ is a surjective submersion and a vector bundle map, it must be fiber-wise surjective, and so does $d^q_{n,k}:V_n\to G_n\times_{G_{n,k}}V_{n,k}$. Then the latter is surjective and i) follows. 
\end{itemize}
\end{proof}

%




Given $q: V\to G$ a higher vector bundle, its {\bf core} $E=\bigoplus_{n\geq0}E_n\to G_0$ is defined by 
$E_n=\ker (d_{n,0}:V_n\to V_{n,0})|_{G_0}$. This graded vector bundle is associated with the restriction to the unit groupoid $V|_M$ via Dold-Kan. Note that $q:V\to G$ is an $N$-vector bundle if and only if $E_n=0_{G_0}$ for all $n>N$, for $\rk E_n= \rk V_n - \rk V_{n,0}$. 
We now introduce the smooth linear versions of cleavages and flatness. Recall that we use the non-standard notation $\sigma_k:[k]\to[n]$ for the inclusion and $s_k:X_n\to X_k$ for the corresponding map on the simplicial set or object.

\begin{definition}\label{def:linear-cleavage}
An {\bf $n$-cleavage} $C_n\subset V_n$ is a set-theoretic $n$-cleavage that is also a sub-bundle, so $C_{n}\oplus \ker d_{n,k}=V_n$ for every $k<n$. A {\bf cleavage} $C=\{C_n\}_{n\geq1}$ is a collection of $n$-cleavages. We say that it is {\bf normal} or {\bf flat} if it is so set-theoretically. And we say that it is {\bf weakly flat} if it is flat over the zero section, namely given $v\in C_{n}$ such that $s_0(v)=0$, $s_k(v)\in C$ for all $k$, and $d_i(v)\in C$ for all $i>0$, then $d_{0}(v)\in C$ as well. 
\end{definition}

\begin{example}\ \label{ex:cleavages}

\begin{enumerate}[a)]
    \item The core of a simplicial vector bundle $q: V\to M$ over a manifold is the graded vector bundle underlying its normalization in the sense of Dold-Kan.
    There is only one possible normal cleavage $C_n=D_n$, spanned by the degenerate simplices, and it is flat, see Lemma \ref{lemma:dk-groupoid}.
    \item The core of a Lie groupoid representation is just the unit vector bundle. Its unique cleavage $C_n=(G\ltimes_R E)_n$ is normal and flat.
    \item The core of a VB-groupoid $q:V\to G$ is given by $E_1=\ker s|_{G_0}$ and $E_0=V_0$. A normal cleavage $C=C_1$ is called a {\it right horizontal lift} in \cite{gsm} and yields a representation up to homotopy $R: G\action E$. Normal cleavages always exist, and any cleavage is weakly flat: if $v\in V_2$ is such that $s_0(v)=0$, $d_1(v)\in C_1$ and $d_2(v)\in C_1$, then $d_1(v)=0$ and $d_2(v)=0$, and therefore $v=0$.  A flat cleavage yields a strict representation and it may not exist, as in the tangent VB-groupoid of the pair groupoid
    $T(S^2\times S^2)\toto T(S^2)$ of the 2-sphere.
\end{enumerate}
\end{example}

We now focus on morphisms of higher vector bundles equipped with cleavages.

\begin{definition}
Let $V$ and $V'$ be higher vector bundles endowed with cleavages $C$ and $C'$, respectively. A morphism $\phi:V\to V'$ is {\bf flat} if  $\phi(C)\subset C'$, and is  {\bf weakly flat} if it satisfies the following more lax condition: if $v\in C\subset V$ is such that $s_0(v)=0$ and $s_k(v)\in C$ for every $k>0$, then $\phi(v)\in C'$. 
\end{definition}

\begin{example}
\begin{enumerate}[a)]
\item Every morphism between simplicial vector bundles over a manifold $M$ must be flat, for the only cleavage is spanned by the degenerate simplices, see Lemma \ref{lemma:dk-groupoid}.
\item Morphisms between 0-vector bundles are also flat, for if $q:V\to G$ is a 0-vector bundle, the only cleavage is $C_n=V_n$, which is normal and flat.
\item Let $V$ and $V'$ be VB-groupoids over a Lie groupoid $G$, endowed with cleavages $C$ and $C'$, respectively. We claim that any morphism $\phi: V\to V'$ of VB-groupoids is weakly flat. In fact, given $v\in C_1$ such that $s_0(v)=0$, we necessarily have $v=0$, so $\phi(v)=0\in C'$.
\end{enumerate}
\end{example}

\begin{remark}\label{rmk:forgetful}
The composition of weakly flat morphisms is weakly flat.  We write $VB^\infty_+(G)$ for the category of higher vector bundles coupled with a normal weakly flat cleavage, and where the morphisms are weakly flat, and $VB^N_+(G)$ for the subcategory of objects of order $N$. We write 
$$U:VB^\infty_+(G)\to VB^\infty(G) \qquad U^N:VB^N_+(G)\to VB^N(G)$$ 
for the forgetful functors. If $G$ is a Lie groupoid then $U^1,U^0$ are equivalences of categories, for every $V$ admits a normal weakly flat cleavage, and every $\phi$ is weakly flat. For $N\geq 2$, we will show with Examples \ref{ex:es} and \ref{ex:not-full} that $U^N$ (and hence $U$) is not essentially surjective nor full. 
\end{remark}


We close this section with two intricate examples. The first one shows a higher vector bundle of order 2 that does not admit a cleavage at all. This shows that $U^2$ is not essentially surjective. Nevertheless, in future work, we will show that every higher vector bundle is homotopic to one admitting a cleavage. The second example present a morphism between higher vector bundles endowed with weakly flat cleavages that is not weakly flat, showing that $U^2$ is not full.

\begin{example}\label{ex:es}
Let $G=S^1$, viewed as a Lie groupoid with a single object $\ast$. Let $V_1\to S^1$ be the non-trivial line bundle. Define $V_n\to G_n$ as the pullback of $\prod_{0\leq i<j\leq n}V_1\to \prod_{0\leq i<j\leq n}G_1$ along the map $\phi: G_n\to\prod_{0\leq i<j\leq n}G_1$, $\phi_{j,i}(g)\mapsto g|_{j,i}$, that sends a simplex to its edges.
It is easy to check that $V$ is a higher vector bundle over $G$ of order 2. 
Then $K_2=\ker(d^q_{2,0}: V_2\to V_{2,0})$ is the pullback $d_0^*V_1$, which is a non-trivial bundle. We will see in Lemma \ref{lemma:direct-sum} that the existence of a cleavage forces $K_2$ to be trivial. Then $q: V\to G$ does not admit a cleavage.
\end{example}


\begin{example}\label{ex:not-full}
Let $(E,\partial)$ be the chain complex such that $E_1=E_2=\R$, $E_n=0$ for $n\neq 1,2$, and $\partial=\id:E_2\to E_1$. Let $G$ be the pair groupoid over $G_0=\{y,x\}$, and let $q:V\to G$ be the pullback of $DK(E,\partial)\to\ast$ along the trivial map $G\to\ast$. Then $q:V\to G$ is a 2-vector bundle and $V$ is a Lie 2-groupoid.
A 2-simplex $v\in V_2$ can be visualized as follows:
$$\begin{matrix}\xymatrix{ 
& v_1 \ar[dl]_{v_{210}+v_{20}-v_{10}} \ar@{}[d]|{v_{210}} & \\
v_2 & & v_0 \ar[ul]_{v_{10}} \ar[ll]^{v_{20}}}\end{matrix}\qquad
\substack{v_i\in\{0_x,0_y\} \\ v_{210}, v_{10},v_{20}\in\R}$$
Four such triangles define a $3$-simplex $v\in V_3$ if their edges agree and the faces satisfy $v_{310}+v_{321}=v_{320}+v_{210}$.
A cleavage $C\subset V$ for $q$ is completely determined by $C_2\subset V_2$. We will define two cleavages $C'_2,C_2\subset V_2$ so that $\id:(V,C')\to (V,C)$ is not weakly flat. The cleavage $C_2\subset V_2$ is the canonical one, consisting of those $v$ with $v_{210}=0$. It is clearly normal and flat. The cleavage $C'_2\subset V_2$ is defined by setting ${C'}_2^g=C_2^g$ whenever $g\notin \{(xxy),(xyx),(yxy)\}\subset G_2$, and setting ${C'}_2^{(xxy)}$, ${C'}_2^{(xyx)}$, ${C'}_2^{(yxy)}$ to be formed by the following type of triangles, respectively: 
$$\begin{matrix}\xymatrix{ 
& 0_x \ar[dl]_{-2\lambda+2\mu} \ar@{}[d]|{-\lambda+\mu} & \\
0_x & & 0_y \ar[ul]_{\lambda} \ar[ll]^{\mu}} & 
& \xymatrix{ 
& 0_y \ar[dl]_{-\lambda-\mu} \ar@{}[d]|{-2\mu} & \\
0_x & & 0_x \ar[ul]_{\lambda} \ar[ll]^{\mu}} & 
& \xymatrix{ 
& 0_x \ar[dl]_{\lambda+\mu} \ar@{}[d]|{2\lambda} & \\
0_y & & 0_y \ar[ul]_{\lambda} \ar[ll]^{\mu}} \\ \\
v_{\lambda,\mu}^{xxy} & & v_{\lambda,\mu}^{xyx} & & v_{\lambda,\mu}^{yxy}\end{matrix}$$
It is tedious but easy to check that $C'$ is normal and weakly flat. 
But $\id:(V,C')\to(V,C)$ is not weakly flat, for $s_0(v^{xxy}_{0,\mu})=0$, $s_k(v^{xxy}_{0,\mu})\in C'_k$ for $k=1,2$, and yet $v^{xxy}_{0,\mu}\notin C_2$.
\end{example}







\section{The semi-direct product construction}
\label{section:sdp}



We present our semi-direct product construction, which starts with a representation up to homotopy $R: G \curvearrowright E$ and produces a well-defined higher vector bundle $(G \ltimes_R E) \to G$.
In the Example \ref{example:sdp2} we show that this recovers our formula for Dold-Kan \ref{prop:inverse-dk} and the Grothendieck Construction \ref{rmk:grothendieck}. We show that $G\ltimes_RE$ comes with a canonical normal weakly flat cleavage.

\


Let $G$ be a higher Lie groupoid. Regard an $n$-simplex $g\in G_n$ as a functor $g:[n]\to G$, 
and denote by $x_i=(\chi_i)^*: G_ n\to  G_ 0$ the {\bf vertex map} induced by pre-composition with the $i$-th inclusion $\chi_i:[0]\to[n]$, $\chi_i(0)=i$. 

\begin{definition}
Given $E=\oplus_{n=0}^N E_n$ a vector bundle over $G_0$, and given $R:G\action E$ a representation up to homotopy, we define the {\bf $n$-simplex vector bundle} $(G\ltimes_R E)_n$ as the vector bundle over $G_n$ given as follows, where the sum is over the injective order maps preserving $0$:
$$(G\ltimes_RE)_n=\bigoplus_{\substack{[k]\xto\alpha[n] \\ \alpha(0)=0}} x_{\alpha(k)}^*E_k$$
\end{definition}

Given $g\in G_n$ a simplex, $\alpha:[k]\to[n]$ an injective order map preserving 0, and $e\in E_k$ a degree $k$ vector, we write $(e,\alpha,g)\in( G\ltimes E)_n$ for the corresponding homogeneous vector in the direct sum. 

\begin{remark}
When $n=2$ we have four injective order maps $\alpha:[k]\to[2]$, namely $\alpha_0$, $\alpha_{10}$, $\alpha_{20}$ and $\alpha_{210}$, where we are writing $\alpha_I$ for the injection with image $I\subset[2]$. The next picture suggests a way to visualize the four types of homogeneous vectors in $(G\ltimes_RE)_2$ over some $g\in G_2$, where the dotted squares represent the 0-section:
$$\begin{matrix}
\begin{tikzpicture}[scale=1]
\draw[dashed](-0.2,0) -- (3.2,0) -- (4.2,1.3) 
node[above] {$(e_0,\alpha_0,g)$}
-- (0.8,1.3) -- (-0.2,0) ;
\draw[fill] (1,1)node[below left] {$g_2$} circle (1.8pt);
\draw[fill] (2,0.5)node[below ] {$g_1$} circle (1.8pt);
\draw[fill] (3,1)node[below right] {$g_0$} circle (1.8pt);
\draw[dashed](1,1) -- (2,0.5) -- (3,1) -- (1,1);
\draw[fill] (1,2) circle (1.8pt);
\draw[fill] (2,1.5) circle (1.8pt);
\draw[fill] (3,2) circle (1.8pt);
\draw[thick](1,2) -- (2,1.5) -- (3,2) -- (1,2);
\end{tikzpicture}
& &
\begin{tikzpicture}[scale=1]
\draw[dashed](-0.2,0) -- (3.2,0) -- (4.2,1.3) 
node[above left] {$(e_1,\alpha_{10},g)$}
-- (0.8,1.3) -- (-0.2,0) ;
\draw[fill] (1,1)node[below left] {$g_2$} circle (1.8pt);
\draw[fill] (2,0.5)node[below ] {$g_1$} circle (1.8pt);
\draw[fill] (3,1)node[below right] {$g_0$} circle (1.8pt);
\draw[dashed](1,1) -- (2,0.5) -- (3,1) -- (1,1);
\draw[fill] (2,1.5) circle (1.8pt);
\draw[thick](1,1) -- (2,1.5) -- (3,1) -- (1,1);
\draw[dashed](2,0.5) -- (2,1.5);
\end{tikzpicture}
\\[10pt]
\begin{tikzpicture}[scale=1]
\draw[dashed](-0.2,0) -- (3.2,0) -- (4.2,1.3) 
node[above left] {$(e_1,\alpha_{20},g)$}
-- (0.8,1.3) -- (-0.2,0) ;
\draw[fill] (1,1)node[below left] {$g_2$} circle (1.8pt);
\draw[fill] (2,0.5)node[below ] {$g_1$} circle (1.8pt);
\draw[fill] (3,1)node[below right] {$g_0$} circle (1.8pt);
\draw[dashed](1,1) -- (2,0.5) -- (3,1) -- (1,1);
\draw[fill] (1,2) circle (1.8pt);
\draw[thick](1,2) -- (2,0.5) -- (3,1) -- (1,2);
\draw[dashed](1,1) -- (1,2);
\end{tikzpicture}
& &
\begin{tikzpicture}[scale=1]
\draw[dashed](-0.2,0) -- (3.2,0) -- (4.2,1.3) 
node[above left] {$(e_2,\alpha_{210},g)$}
-- (0.8,1.3) -- (-0.2,0) ;
\draw[fill] (1,1)node[below left] {$g_2$} circle (1.8pt);
\draw[fill] (2,0.5)node[below ] {$g_1$} circle (1.8pt);
\draw[fill] (3,1)node[below right] {$g_0$} circle (1.8pt);
\draw[dashed](1,1) -- (2,0.5) -- (3,1) -- (1,1);
\draw[thick](1,1) -- (3,1) -- (2,0.5);
\draw[thick]  (1,1) to[out=90,in=90,distance=1cm] (2,0.5);
\draw[thick] (3,1) -- (1.5,1.5);
\end{tikzpicture}
\end{matrix}
$$
The vector $(e_0,\alpha_0,g)$ can be seen as a triangle with first vertex $e_0\in E^{g_0}_0$, horizontal faces over $g_{10}$ and $g_{20}$, and a thin filling of that horn over $g$. 
The vector $(e_1,\alpha_{10},g)$ has its first vertex on the 0-section, a non-horizontal face over $g_{10}$ ruled by $e_1\in E^{g_1}_1$, a null face over $g_{20}$, and a thin filling of that horn over $g$. 
Then $(e_1,\alpha_{20},g)$ is analogous to the previous one, with a null face over $g_{10}$, a non-horizontal face over $g_{20}$ ruled by $e_1\in E^{g_2}_1$, and a thin filling. 
Finally, the vector $(e_2,\alpha_{210},g)$ has its 0-horn on the zero section, and its filling is not thin, but curved by $e_2\in E^{g_2}_2$.
\end{remark}


\begin{example}\label{example:sdp1}
Let $G$ be a Lie groupoid. We revisit the main examples discussed in \ref{ex:ruth}:

\begin{enumerate}[a)]
    \item When $G=M$ is the unit groupoid of a manifold, so $R:M \action E$ is a chain complex, we have $G_n=M$ for every $n$, every vertex map is the identity, and the $n$-simplex vector bundle is isomorphic to the $n$-th degree of our inverse of Dold-Kan \ref{prop:inverse-dk}:
    $$(M\ltimes_RE)_n\cong DK(E,\partial)_n$$
    \item When $N=0$, so $E=E_0$ and $R$ is a usual representation, the $n$-simplex vector bundle is isomorphic to the $n$-th degree of the nerve of the translation groupoid $Tr(R)$:
    $$(G\ltimes_RE)_n=x_0^*E\cong Tr(R)_n$$
    A vector in $(G\ltimes_RE)_n$ can be seen as a tuple $((g_n,\dots,g_1),e)$, where $e\in x_0^*E$ and $g=(g_n,\dots,g_1)\in G_n$ is a chain of composable arrows. Such a vector is mapped to the chain $(g_n,R^{g_{n-1}\dots g_1}e)\dots(g_2,R^{g_1}e)(g_1,e)$ of composable arrows in $Tr(R)_n$.
    \item When $N=1$, so $E=E_1\oplus E_0$, the $n$-simplex vector bundle is isomorphic to the $n$-th degree of the Grothendieck construction $Gr(R)_n$ recalled in Remark \ref{rmk:grothendieck}:
    $$(G\ltimes_RE)_n= x_0^*E_0\oplus\bigoplus_{i=1}^nx_i^*E_1\cong Gr(R)_n$$ 
    A vector in $(G\ltimes_RE)_n$ can be seen as a tuple
    $(e,c_n,\dots,c_1,g)$, where $e\in x_0^*E$, $c_i\in x_i^*E_1$, and $g=(g_n,\dots,g_1)\in G_n$ is a chain of composable arrows. Such a vector is mapped to the collection $((c_n,g_n\dots g_1,e),\dots,(c_2,g_2g_1,e),(c_1,g_1,e))$,  
    where we regard a point in $Gr(R)_n$ as $n$ arrows with the same source. This {\it change of variables}, from composable arrows to arrows with the same source, was discussed in \cite[Rmk. 11.13]{cms}, and turns out to be very convenient when working with higher categories.
\end{enumerate}
\end{example}


We have just constructed the vector bundles $(G\ltimes_RE)_n$. Let us provide formulas for the faces and degeneracies, extending those of our Dold-Kan construction in Proposition \ref{prop:inverse-dk}. Given $\theta:[k]\to[n]$, recall that $\theta':[k+1]\to[n+1]$ is the map given by $\theta'(0)=0$, $\theta'(i+1)=\theta(i)+1$, and that we use $\sigma_k:[k]\to[n]$ and $\tau_k:[k]\to[n]$ for the first and last inclusions. Given $\beta:[l]\to[n]$, we write $\pi_\beta:(G\ltimes E)_n\to E_l$ for the projection covering $x_{\beta(l)}: G_n\to G_0$. 

\begin{definition}\label{def:sdp-faces-degeneracies}
We define the faces $d_i:( G\ltimes_R E)_n\to( G\ltimes_R E)_{n-1}$ and the degeneracies
$u_j:( G\ltimes_R E)_n\to( G\ltimes_R E)_{n+1}$ of the semi-direct product as follows:
\begin{itemize}
    \item $\pi_\beta u_j=\pi_{\upsilon_j\beta}$ for every $\beta:[l]\to[n+1]$, with the convention that $\pi_\gamma=0$ whenever $\gamma$ is not injective; 
    \item $\pi_\beta d_i=\pi_{\delta_i\beta}$ for every 
    $\beta:[l]\to[n-1]$ and $i>0$; and
    \item 
$\pi_\beta d_0=
\sum_{k=0}^{l+1}(-1)^{l}
R_{l+1-k}^{g\beta'\tau_{l+1-k}}\pi_{\beta'\sigma_{k}}-\sum_{i=1}^{l} (-1)^i\pi_{\beta'\delta_i}$ for every $\beta:[l]\to[n-1]$.
\end{itemize}
\end{definition}

Let us rewrite the above faces and degeneracies, by computing their values on homogeneous vectors, as we did in Remark \ref{rmk:homog-formulas} for the inverse of Dold-Kan. 

\begin{remark}\label{rmk:homog-formulas-2}\label{rmk:alt-formulas}
The positive faces and the degeneracies on the homogeneous vectors are given by the formulas below:
$$
\pi_\beta u_j(e,\alpha,g)=\begin{cases} 
e & \alpha=\upsilon_j\beta \\
0 & \text{otherwise}
\end{cases}
\qquad
\pi_\beta d_i(e,\alpha,g)=\begin{cases}
  e & \alpha=\delta_i\beta  \\ 
0 & \text{otherwise}\end{cases}\ (i\neq0)
$$
Regarding the zero face map $d_0$, there are two cases where the contributions are non-trivial, which we visualize using black dots for $\alpha$ and gray dots for $\delta_0\beta$:
\begin{enumerate}[(I)]
        \item If $\alpha=\beta'\sigma_k$ for some $0\leq k\leq l+1$, then  
$\pi_\beta d_0(e,\alpha,g)=(-1)^{l} R^{g\beta'\tau_{l+1-k}}_{l+1-k}(e)$, so the vector $e$ sitting over the first $k$ arrows $g\beta'\sigma_k$ is pushed forward along the last $l+1-k$ arrows $g\beta'\tau_{l+1-k}$:
$$\begin{tikzpicture}
\draw[thick](6,0)node[below] {$0$} --(0,0) node[below] {$n$};
\draw[fill] (6,0.2) circle (1.8pt);
\draw (6.4,0.2) node {\footnotesize $\alpha$};
\draw[fill] (5.5,0.2) circle (1.8pt);
\draw[fill] (4.5,0.2) circle (1.8pt);
\draw[fill] (4,0.2) circle (1.8pt);
\draw (5.5,0) node[below] {$1$};
\draw (6,0.45) circle (1.8pt);
\draw (6.4,0.45) node {\footnotesize $\beta'$};
\draw[fill,gray] (5.5,0.45) circle (1.8pt);
\draw[fill,gray] (4.5,0.45) circle (1.8pt);
\draw[fill,gray] (4,0.45) circle (1.8pt);
\draw[fill,gray] (3.5,0.45) circle (1.8pt);
\draw (4,0) node[below] {\footnotesize $\alpha(k)$};
\draw[fill,gray] (3,0.45) circle (1.8pt);
\draw[fill,gray] (2,0.45) circle (1.8pt); \draw (2,0) node[below] {\footnotesize $\beta'(l+1)$};
\end{tikzpicture}$$
        \item If $\alpha=\beta'\delta_i$ for some $1\leq i\leq k=l$, then $\pi_\beta d_0(e,\alpha,g)=(-1)^{i+1}e$:
$$\begin{tikzpicture}
\draw[thick](6,0)node[below] {$0$} --(0,0) node[below] {$n$};
\draw[fill] (6,0.2) circle (1.8pt);
\draw (6.4,0.2) node {\footnotesize $\alpha$};
\draw[fill] (5.5,0.2) circle (1.8pt);
\draw[fill] (4.5,0.2) circle (1.8pt);
\draw[fill] (4,0.2) circle (1.8pt);
\draw[fill] (3,0.2) circle (1.8pt);
\draw[fill] (2,0.2) circle (1.8pt); 
\draw (6,0.45) circle (1.8pt);
\draw (6.4,0.45) node {\footnotesize $\beta'$};
\draw[fill,gray] (5.5,0.45) circle (1.8pt);
\draw[fill,gray] (4.5,0.45) circle (1.8pt);
\draw[fill,gray] (4,0.45) circle (1.8pt);
\draw[fill,gray] (3.5,0.45) circle (1.8pt);
\draw[fill,gray] (3,0.45) circle (1.8pt);
\draw[fill,gray] (2,0.45) circle (1.8pt);
\draw (5.5,0) node[below] {$1$};
\draw (3.5,0) node[below] {\footnotesize $\beta'(i)$};
\draw (2,0) node[below] {\footnotesize $\alpha(k)$};
\end{tikzpicture}$$
\end{enumerate}
We can derive from the above formulas the following alternative useful computations:
\begin{enumerate}[$\bullet$]
    \item The positive faces are given by $d_i(e,\alpha,g)=
\begin{cases}(e,\upsilon_i\alpha,g\delta_i) & i\notin\alpha \\ 0 & i\in\alpha\end{cases}$
    \item The degeneracies are given by $u_j(e,\alpha,g)=\begin{cases}
(e,\delta_{1}\alpha,g\upsilon_0) & j=0\\
(e,\delta_j\alpha,g\upsilon_j)& j\notin\alpha\\
(e,\delta_j\alpha,g\upsilon_j)+(e,\delta_{j+1}\alpha,g\upsilon_j)& 0\neq j\in\alpha
\end{cases}$
    \item The 0-face is $d_0(e,\alpha,g)=
\sum_{\beta'\sigma_k=\alpha}(-1)^{l}(R_{l+1-k}^{g\beta'\tau_{l+1-k}}(e),\beta,g\delta_0)-\sum_{\beta'\delta_i=\alpha}(-1)^i(e,\beta,g\delta_0)$ where the sums are over all $\beta:[l]\to[n-1]$ satisfying the equations, and $1\leq i\leq l$.
\end{enumerate}
Note in particular that if $\alpha\neq\sigma_0$ and $1\notin\alpha$ then $d_0(v,\alpha,g)=(v,\upsilon_0\alpha,g\delta_0)$.
\end{remark}



We next use the above formulas to check that the faces and degeneracies of the semi-direct product satisfy the simplicial identities, hence defining a simplicial vector bundle.

\begin{proposition}\label{prop:simplicial-identities}
Given $G$ a higher Lie groupoid, $E=\oplus_{n=0}^NE_n\to G_0$ a graded vector bundle and $R: G\action E$ a representation up to homotopy,
the operators $d_i$ and $u_j$ from Definition \ref{def:sdp-faces-degeneracies} satisfy the simplicial identities recalled in Remark \ref{rmk:simplicial-identities}.    
\end{proposition}

The proof is rather technical and long, so we split it into several steps.

\begin{proof}[Proof of the simplicial identities
not involving $d_0$]
Given an order map $[m]\xto\theta[n]$ preserving $0$, define $\theta^*:( G\ltimes E)_n\to ( G\ltimes E)_m$, $\pi_\beta \theta^*=\pi_{\theta\beta}$,
always with the convention that $\pi_\gamma=0$ whenever $\gamma$ is not injective. The positive faces and the degeneracies are instances of this formula. Given $[m]\xto{\theta_1}[n]\xto{\theta_2}[p]$ preserving $0$, we have $\pi_\beta\theta_1^*\theta_2^*=\pi_{\theta_1\beta}\theta_2^*=\pi_{\theta_2\theta_1\beta}=\pi_\beta (\theta_2\theta_1)^*:(G\ltimes E)_p\to(G\ltimes E)_m$. This functorial property readily implies every simplicial identity not involving $d_0$.
\end{proof}

\begin{proof}[Proof of
$d_0u_j=u_{j-1}d_0$]
By definition we have 
\begin{align*}
 \pi_\beta d_0u_j &=\left(
\sum_{k=0}^{l+1} (-1)^{l} R_{l+1-k}^{g\upsilon_j\beta'\tau_{l+1-k}}\pi_{\beta'\sigma_k}
-\sum_{i=1}^{l} (-1)^i\pi_{\beta'\delta_i}\right)u_j\\
&= \sum_{k=0}^{l+1}(-1)^{l} R_{l+1-k}^{g\upsilon_j\beta'\tau_{l+1-k}}\pi_{\upsilon_j\beta'\sigma_k}-
\sum_{i=1}^l (-1)^i\pi_{\upsilon_j\beta'\delta_i}.
\end{align*}
The other composition is
$$\pi_\beta u_{j-1}d_0 = \pi_{\upsilon_{j-1}\beta}d_0=\sum_{k=0}^{l+1} (-1)^{l} R_{l+1-k}^{g(\upsilon_{j-1}\beta)'\tau_{l+1-k}}\pi_{(\upsilon_{j-1}\beta)'\sigma_k}-
\sum_{i=1}^{l} (-1)^i\pi_{(\upsilon_{j-1}\beta)'\delta_i}.$$
These two expressions agree because 
$(\theta_1\theta_2)'=\theta_1'\theta_2'$ and $\upsilon_j'=\upsilon_{j+1}$.
\end{proof}

\begin{proof}[Proof of 
$d_0d_j=d_{j-1}d_0$ for $j>1$]
These are completely analogous to the previous case, now using that $\delta_i'=\delta_{i+1}$.
\end{proof}

The previous simplicial identities follow from the combinatorial definition, disregarding RH1) and RH2). RH1) is needed now to show $d_0u_0=\id$, and RH2) will be equivalent to $d_0d_1=d_0d_0$.

\begin{proof}[Proof of $d_0u_0=\id$]
Given $\beta:[l]\to[n]$, we have
\begin{align*}
\pi_\beta d_0u_0  &=
\left(\sum_{k=0}^{l+1}(-1)^{l} R_{l+1-k}^{g\upsilon_0\beta'\tau_{l+1-k}}\pi_{\beta'\sigma_k}-
\sum_{i=1}^l (-1)^{i}\pi_{\beta'\delta_i} 
\right) u_0\\
&=\sum_{k=0}^{l+1} (-1)^{l} R_{l+1-k}^{g\upsilon_0\beta'\tau_{l+1-k}}\pi_{\upsilon_0\beta'\sigma_k}-
\sum_{i=0}^l (-1)^{i}\pi_{\upsilon_0\beta'\delta_i}
\end{align*}
Note that $\upsilon_0\beta'\sigma_k$ is not injective if $k>0$ and $\upsilon_0\beta'\delta_i$ is not injective if $i>1$. 
If $l=0$, then $\beta=\sigma_0$ and
$\pi_\beta d_0u_0=
R^{g\upsilon_0\beta'\tau_1}_{1}\pi_{\upsilon_0\beta'\sigma_0}=
\pi_{\beta}$,
for $g\upsilon_0\beta'\tau_1=g\upsilon_0\sigma_2$ is an identity and $R$ satisfies RH1).
If $l>0$ then 
$\pi_\beta d_0u_0=
R_{l+1}^{g\upsilon_0\beta'}\pi_{\upsilon_0\beta'\sigma_0}+\pi_{\upsilon_0\beta'\delta_1}=\pi_\beta$, for $\upsilon_0\beta'\delta_1=\beta$, the first edge of 
$g\upsilon_0\beta'$ is degenerate, and $R$ satisfies RH1).
\end{proof}

\begin{proof}[Proof of $d_0d_1=d_0d_0$]
We show the identity for homogeneous vectors, namely $\pi_\beta d_0d_0(e,\alpha,g)=\pi_\beta d_0d_1(e,\alpha,g)$. 
We split it into several cases, depending on how $\alpha$ and $\beta$ relate.

\begin{enumerate}[(i)]
    \item {\bf $\alpha=\beta''\sigma_k$, $k\geq 1$:} 
In this case $1\in\alpha$ and $d_1(e,\alpha,g)=0$, so we must prove that $\pi_\beta d_0d_0(e,\alpha,g)=0$. 
We can write $\pi_\beta d_0d_0(e,\alpha,g)$
as the sum of the terms of type I-I, where some points on the left are added to the index twice; plus the sum of the terms of type I-II, where first every dot but one is added and then the remaining one. 
$$
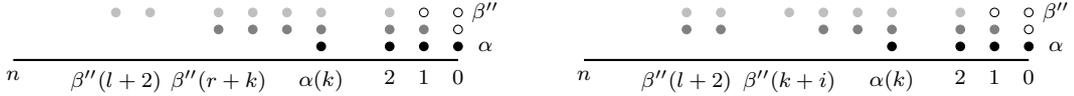
\begin{tikzpicture}[scale=.9]
\draw[thick](6.5,0)node[below] {\scriptsize$0$} --(0,0) node[below] {\scriptsize$n$};
\draw[fill] (6.5,0.2) circle (1.8pt);
\draw[fill] (6,0.2) circle (1.8pt);
\draw (6.9,0.2) node {\scriptsize $\alpha$};
\draw[fill] (5.5,0.2) circle (1.8pt);
\draw[fill] (4.5,0.2) circle (1.8pt);
\draw (6.5,0.45) circle (1.8pt);
\draw[fill,gray] (6,0.45) circle (1.8pt);
\draw (6.9,0.45) ;
\draw[fill,gray] (5.5,0.45) circle (1.8pt);
\draw[fill,gray] (4.5,0.45) circle (1.8pt);
\draw[fill,gray] (3.5,0.45) circle (1.8pt);
\draw[fill,gray] (3,0.45) circle (1.8pt);
\draw[fill,gray] (4,0.45) circle (1.8pt);
\draw (6.5,0.7) circle (1.8pt);
\draw (6,0.7) circle (1.8pt);
\draw (6.9,0.7) node {\scriptsize $\beta''$};
\draw[fill,lightgray] (5.5,0.7) circle (1.8pt);
\draw[fill,lightgray] (4.5,0.7) circle (1.8pt);
\draw[fill,lightgray] (4,0.7) circle (1.8pt);
\draw[fill,lightgray] (3.5,0.7) circle (1.8pt);
\draw[fill,lightgray] (3,0.7) circle (1.8pt);
\draw[fill,lightgray] (2,0.7) circle (1.8pt);
\draw[fill,lightgray] (1.5,0.7) circle (1.8pt);
\draw (6,0) node[below] {\scriptsize $1$};
\draw (5.5,0) node[below] {\scriptsize$2$};
\draw (3,0) node[below] {\scriptsize $\beta''(r+k)$};
\draw (4.5,0) node[below] {\scriptsize $\alpha(k)$};
\draw (1.5,0) node[below] {\scriptsize $\beta''(l+2)$};
\end{tikzpicture}\qquad 
\begin{tikzpicture}[scale=.9]
\draw[thick](6.5,0)node[below] {\scriptsize$0$} --(0,0) node[below] {\scriptsize$n$};
\draw[fill] (6.5,0.2) circle (1.8pt);
\draw[fill] (6,0.2) circle (1.8pt);
\draw (6.9,0.2) node {\scriptsize $\alpha$};
\draw[fill] (5.5,0.2) circle (1.8pt);
\draw[fill] (4.5,0.2) circle (1.8pt);
\draw (6.5,0.45) circle (1.8pt);
\draw[fill,gray] (6,0.45) circle (1.8pt);
\draw[fill,gray] (5.5,0.45) circle (1.8pt);
\draw[fill,gray] (4.5,0.45) circle (1.8pt);
\draw[fill,gray] (4,0.45) circle (1.8pt);
\draw[fill,gray] (3.5,0.45) circle (1.8pt);
\draw[fill,gray] (2,0.45) circle (1.8pt);
\draw[fill,gray] (1.5,0.45) circle (1.8pt);
\draw (6.5,0.7) circle (1.8pt);
\draw (6,0.7) circle (1.8pt);
\draw (6.9,0.7) node {\scriptsize $\beta''$};
\draw[fill,lightgray] (5.5,0.7) circle (1.8pt);
\draw[fill,lightgray] (4.5,0.7) circle (1.8pt);
\draw[fill,lightgray] (4,0.7) circle (1.8pt);
\draw[fill,lightgray] (3.5,0.7) circle (1.8pt);
\draw[fill,lightgray] (3,0.7) circle (1.8pt);
\draw[fill,lightgray] (2,0.7) circle (1.8pt);
\draw[fill,lightgray] (1.5,0.7) circle (1.8pt);
\draw (6,0) node[below] {\scriptsize$1$};
\draw (5.5,0) node[below] {\scriptsize$2$};
\draw (3,0) node[below] {\scriptsize $\beta''(k+i)$};
\draw (4.5,0) node[below] {\scriptsize $\alpha(k)$};
\draw (1.5,0) node[below] {\scriptsize $\beta''(l+2)$};
\end{tikzpicture}$$
Writing $m=l+2-k$ and $\tilde g=g\beta''\tau_m$, 
and using the identities
$\delta_0\beta'\tau_{m}=\beta''\tau_{m}$, $\sigma_{ r+k}\tau_{r}=\tau_{m}\sigma_{r}$, and $\delta_{i+1}\tau_{m-1}=\tau_{m}\delta_{i+1-k}$, 
the sums of the terms of type I-I and I-II are, respectively, 
$$
\sum_{r=0}^{m} (-1)^{l} R_{m-r}^{\tilde g\tau_{m-r}} (-1)^{k+r-1} R_{r}^{\tilde g\sigma_{ r}}(e) 
\quad\text{ and }\quad
\sum_{i=1}^{m-1} (-1)^{i+k}(1)^l R_{m-1}^{\tilde g\delta_{i}}(e),
$$
and the two sums cancel out by axiom RH2) of a representation up to homotopy. 

\item $\alpha=\beta''\sigma_0$: 
In this case $k=0$ and $1\notin\alpha=\sigma_0$.
Writing $\tilde g=g\beta''$, $m=l+2$ and using that $\delta_1\beta'\tau_{l+1}=\delta_1\beta'=\beta''\delta_1$ we get $\pi_\beta d_0d_1(e,\sigma_0,g)=\pi_\beta d_0(e,\sigma_0,g\delta_1)=
(-1)^l R_{m-1}^{\tilde g\delta_1}(e)$. 
The computation of $\pi_\beta d_0d_0 (e,\sigma_0,g)$ is similar as before, we group the terms of type I-I on one side, just that $r\neq0$, for we must add up at least one dot, and we group the terms of type I-II on the other side, just that $i\neq1$, for the missing dot cannot be at position 1. 
The missing term of type I-I vanishes for $\deg R^{\tilde g\sigma_0}_0(e)=-1$, and the missing term of type I-II is exactly $-\pi_\beta d_0d_1(e,\sigma_0,g)$, so the identity follows.

\item $\alpha=\beta''\sigma_{k+1}\delta_i=\beta''\delta_i\sigma_k$, $1<i\leq k$: 
Again $1\in\alpha$, $\pi_\beta d_1(e,\alpha,g)=0$, and we have to show that $\pi_\beta d_0d_0(e,\alpha,g)=0$.
In this case, there is only one term of type I-II and only one of type II-I.
They only differ on the sign and therefore they cancel out.
$$
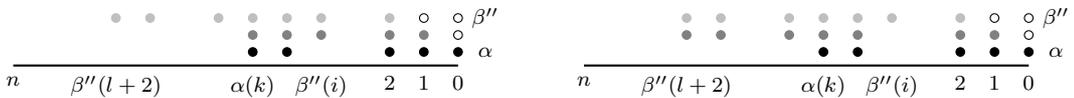
\begin{tikzpicture}[scale=0.9]
\draw[thick](6.5,0)node[below] {\scriptsize $0$} --(0,0) node[below] {\scriptsize $n$};
\draw[fill] (6.5,0.2) circle (1.8pt);
\draw[fill] (6,0.2) circle (1.8pt);
\draw (6.9,0.2) node {\scriptsize $\alpha$};
\draw[fill] (5.5,0.2) circle (1.8pt);
\draw[fill] (4.0,0.2) circle (1.8pt);
\draw[fill] (3.5,0.2) circle (1.8pt);
\draw (6.5,0.7) circle (1.8pt);
\draw (6,0.7) circle (1.8pt);
\draw (6.9,0.7) node {\scriptsize $\beta''$};
\draw[fill,lightgray] (5.5,0.7) circle (1.8pt);
\draw[fill,lightgray] (4.5,0.7) circle (1.8pt);
\draw[fill,lightgray] (4,0.7) circle (1.8pt);
\draw[fill,lightgray] (3.5,0.7) circle (1.8pt);
\draw[fill,lightgray] (3,0.7) circle (1.8pt);
\draw[fill,lightgray] (2,0.7) circle (1.8pt);
\draw[fill,lightgray] (1.5,0.7) circle (1.8pt);
\draw (6.5,0.45) circle (1.8pt);
\draw[fill,gray] (6,0.45) circle (1.8pt);
\draw[fill,gray] (5.5,0.45) circle (1.8pt);
\draw[fill,gray] (4.5,0.45) circle (1.8pt);
\draw[fill,gray] (4,0.45) circle (1.8pt);
\draw[fill,gray] (3.5,0.45) circle (1.8pt);
\draw (6,0) node[below] {\scriptsize$1$};
\draw (5.5,0) node[below] {\scriptsize$2$};
\draw (3.5,0) node[below] {\scriptsize $\alpha(k)$};
\draw (4.5,0) node[below] {\scriptsize $\beta''(i)$};
\draw (1.5,0) node[below] {\scriptsize $\beta''(l+2)$};
\end{tikzpicture}
\qquad 
\begin{tikzpicture}[scale=0.9]
\draw[thick](6.5,0)node[below] {\scriptsize $0$} --(0,0) node[below] {\scriptsize $n$};
\draw[fill] (6.5,0.2) circle (1.8pt);
\draw[fill] (6,0.2) circle (1.8pt);
\draw (6.9,0.2) node {\scriptsize $\alpha$};
\draw[fill] (5.5,0.2) circle (1.8pt);
\draw[fill] (4.0,0.2) circle (1.8pt);
\draw[fill] (3.5,0.2) circle (1.8pt);
\draw (6.5,0.45) circle (1.8pt);
\draw[fill,gray] (6,0.45) circle (1.8pt);
\draw[fill,gray] (5.5,0.45) circle (1.8pt);
\draw[fill,gray] (3,0.45) circle (1.8pt);
\draw[fill,gray] (4,0.45) circle (1.8pt);
\draw[fill,gray] (3.5,0.45) circle (1.8pt);
\draw[fill,gray] (2,0.45) circle (1.8pt);
\draw[fill,gray] (1.5,0.45) circle (1.8pt);
\draw (6.5,0.7) circle (1.8pt);
\draw (6,0.7) circle (1.8pt);
\draw (6.9,0.7) node {\scriptsize $\beta''$};
\draw[fill,lightgray] (5.5,0.7) circle (1.8pt);
\draw[fill,lightgray] (4.5,0.7) circle (1.8pt);
\draw[fill,lightgray] (4,0.7) circle (1.8pt);
\draw[fill,lightgray] (3.5,0.7) circle (1.8pt);
\draw[fill,lightgray] (3,0.7) circle (1.8pt);
\draw[fill,lightgray] (2,0.7) circle (1.8pt);
\draw[fill,lightgray] (1.5,0.7) circle (1.8pt);
\draw (6,0) node[below] {\scriptsize 1};
\draw (5.5,0) node[below] {\scriptsize 2};
\draw (3.5,0) node[below] {\scriptsize $\alpha(k)$};
\draw (4.5,0) node[below] {\scriptsize $\beta''(i)$};
\draw (1.5,0) node[below] {\scriptsize $\beta''(l+2)$};
\end{tikzpicture}
$$

\item $\alpha=\beta''\delta_j\delta_i=\beta''\delta_{i}\delta_{j-1}$, $1<i<j\leq k+1$: As in the previous case, $1\in \alpha$, $\pi_\beta d_1(e,\alpha,g)=0$, and $\pi_\beta d_0d_0(e,\alpha,g)$ has exactly two non-vanishing contributions of type II-II, corresponding to add up first one of the missing dots and then the other. They agree except for the sign and they cancel each other.
$$
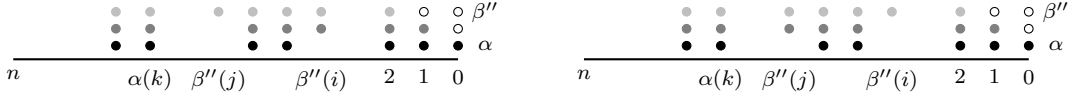
\begin{tikzpicture}[scale=0.9]
\draw[thick](6.5,0)node[below] {\scriptsize $0$} --(0,0) node[below] {\scriptsize $n$};
\draw[fill] (6.5,0.2) circle (1.8pt);
\draw[fill] (6,0.2) circle (1.8pt);
\draw (6.9,0.2) node {\scriptsize $\alpha$};
\draw[fill] (5.5,0.2) circle (1.8pt);
\draw[fill] (4.0,0.2) circle (1.8pt);
\draw[fill] (3.5,0.2) circle (1.8pt);
\draw[fill] (2,0.2) circle (1.8pt); 
\draw[fill] (1.5,0.2) circle (1.8pt); 
\draw (6.5,0.7) circle (1.8pt);
\draw (6,0.7) circle (1.8pt);
\draw (6.9,0.7) node {\scriptsize $\beta''$};
\draw[fill,lightgray] (5.5,0.7) circle (1.8pt);
\draw[fill,lightgray] (4.5,0.7) circle (1.8pt);
\draw[fill,lightgray] (4,0.7) circle (1.8pt);
\draw[fill,lightgray] (3.5,0.7) circle (1.8pt);
\draw[fill,lightgray] (3,0.7) circle (1.8pt);
\draw[fill,lightgray] (2,0.7) circle (1.8pt);
\draw[fill,lightgray] (1.5,0.7) circle (1.8pt);
\draw (6.5,0.45) circle (1.8pt);
\draw[fill,gray] (6,0.45) circle (1.8pt);
\draw[fill,gray] (5.5,0.45) circle (1.8pt);
\draw[fill,gray] (4.5,0.45) circle (1.8pt);
\draw[fill,gray] (4,0.45) circle (1.8pt);
\draw[fill,gray] (3.5,0.45) circle (1.8pt);
\draw[fill,gray] (2,0.45) circle (1.8pt);
\draw[fill,gray] (1.5,0.45) circle (1.8pt);
\draw (6,0) node[below] {\scriptsize$1$};
\draw (5.5,0) node[below] {\scriptsize$2$};
\draw (3,0) node[below] {\scriptsize $\beta''(j)$};
\draw (4.5,0) node[below] {\scriptsize $\beta''(i)$};
\draw (2,0) node[below] {\scriptsize $\alpha(k)$};
\end{tikzpicture}
\qquad 
\begin{tikzpicture}[scale=0.9]
\draw[thick](6.5,0)node[below] {\scriptsize $0$} --(0,0) node[below] {\scriptsize $n$};
\draw[fill] (6.5,0.2) circle (1.8pt);
\draw[fill] (6,0.2) circle (1.8pt);
\draw (6.9,0.2) node {\scriptsize $\alpha$};
\draw[fill] (5.5,0.2) circle (1.8pt);
\draw[fill] (4.0,0.2) circle (1.8pt);
\draw[fill] (3.5,0.2) circle (1.8pt);
\draw[fill] (2,0.2) circle (1.8pt); 
\draw[fill] (1.5,0.2) circle (1.8pt); 
\draw (6.5,0.7) circle (1.8pt);
\draw (6,0.7) circle (1.8pt);
\draw (6.9,0.7) node {\scriptsize $\beta''$};
\draw[fill,lightgray] (5.5,0.7) circle (1.8pt);
\draw[fill,lightgray] (4.5,0.7) circle (1.8pt);
\draw[fill,lightgray] (4,0.7) circle (1.8pt);
\draw[fill,lightgray] (3.5,0.7) circle (1.8pt);
\draw[fill,lightgray] (3,0.7) circle (1.8pt);
\draw[fill,lightgray] (2,0.7) circle (1.8pt);
\draw[fill,lightgray] (1.5,0.7) circle (1.8pt);
\draw (6.5,0.45) circle (1.8pt);
\draw[fill,gray] (6,0.45) circle (1.8pt);
\draw[fill,gray] (5.5,0.45) circle (1.8pt);
\draw[fill,gray] (4,0.45) circle (1.8pt);
\draw[fill,gray] (3.5,0.45) circle (1.8pt);
\draw[fill,gray] (3,0.45) circle (1.8pt);
\draw[fill,gray] (2,0.45) circle (1.8pt);
\draw[fill,gray] (1.5,0.45) circle (1.8pt);
\draw (6,0) node[below] {\scriptsize$1$};
\draw (5.5,0) node[below] {\scriptsize$2$};
\draw (3,0) node[below] {\scriptsize $\beta''(j)$};
\draw (4.5,0) node[below] {\scriptsize $\beta''(i)$};
\draw (2,0) node[below] {\scriptsize $\alpha(k)$};
\end{tikzpicture}$$

\item $\alpha=\beta''\sigma_{k+1}\delta_1$, $k\geq 1$
or $\alpha=\beta''\delta_j\delta_1$, $k>0$:
In these cases $\alpha(1)>1$, $\alpha\neq0$ and we have
$\pi_\beta d_0d_1(e,\alpha,g)=\pi_\beta d_0(e,\sigma_0\alpha,g\delta_0)$ and $\pi_\beta d_0d_0(e,\alpha,g)=\pi_\beta d_0(e,\sigma_0\alpha,g\delta_1)$. These expressions agree because, by definition, $\pi_\beta d_0$ only depend on $g\tau_{m-k}$ rather than $g$.\qedhere
\end{enumerate}
\end{proof}


We are ready to show that the semi-direct product $(G\ltimes_,d_i,u_j)$ is a higher vector bundle equipped with a canonical cleavage. 
Recall that $\iota_n:[n]\to[n]$ denotes the identity. We write $C^{can}\subset (G\ltimes_R E)_n$ for the sub-bundle of vectors whose $\iota_n$ component vanishes:
$$C^{can}_n=\{v:\pi_{\iota_n}(v)=0\}\subset (G\ltimes_R E)_n= \bigoplus_{\substack{[k]\xto\alpha[n] \\ \alpha(0)=0}} x_{\alpha(k)}^*E_k$$

\begin{theorem}\label{thm:sdp}
Given $G$ a higher Lie groupoid and $R: G\action E=\oplus_{n=0}^N E_n$ a  representation up to homotopy, the semi-direct product $q: G\ltimes_R E\to G$ is a higher vector bundle of order $N$ and core $E$, and $C^{can}_n\subset (G\ltimes_R E)_n$ is a normal weakly flat cleavage.
\end{theorem}

\begin{proof}
We have shown that $q$ is a simplicial vector bundle in Proposition 
 \ref{prop:simplicial-identities}.
Let us show now that the relative horn maps $d_{n,k}^q:(G\ltimes_RE)_n\to d_{n,k}^*(G\ltimes_RE)_{n,k}$ are fiber-wise surjective over $G_n$. We have that $d_{1,0}^q:d_0^*E_1\oplus d_1^*E_0\to d_1^*E_0$ is just the projection and is therefore surjective. 
And $d_{1,1}^q:d_0^*E_1\oplus d_1^*E_0\to d_0^*E_0$,  given by $(e_1,10,g)+(e_0,0,g)\mapsto (R^{t(g)}_0(e_1)+R^g_1(e_0),0,t(g))$,
is also fiber-wise surjective, for $R^g_1$ is a quasi-isomorphism.

Suppose now that $G\ltimes_RE$ is not a higher vector bundle, and take $d_{n,k}^q$ a relative horn map that is not surjective, with $n>1$ minimal.
It follows from the inductive argument in Lemma \ref{lemma:smooth-horns} that $(G\ltimes_RE)_{n,k}$ is a well-defined vector bundle of rank $\rk(G\ltimes_RE)_n-\rk E_n$. Let us show that $\ker d_{n,k}^q$ has fiber-wise dimension $\rk E_n$, which contradicts that 
$d^q_{n,k}$ is not surjective.

If $v\in \ker d_{n,k}^q$ then $d_i(v)=0$ for $i\neq k$ and $v$ is supported on the indices containing $\delta_k\subset[n]$. 
If $k=0$ then $\ker d^q_{n,0}=x_n^*E_n$, for $v\in\ker d^q_{n,0} $ must be homogeneous of type $\iota_n$. 
If $0<k<n$ then a vector in $\ker d^q_{n,k}$ is of the form $v=(e,\iota_n,g)+(e',\delta_k,g)$ and we have 
    $0=\pi_{\iota_{n-1}}d_0(v)=(-1)^{k+1}e'+(-)^{n-1} R_0(e)=0$, so $e\in x_n^*E_n$ is arbitrary and determines $e'$ completely.
In either case, we have $\rk\ker d_{n,k}^q=\rk E_n$. It remains to consider the case $k=n$. 

To ease the notations, write $y'\xfrom g' x'=t_1(g)=g\tau_1$, and consider $v=(e,\iota_n,g)+(e',\delta_n,g)$ a vector in $\ker d^q_{n,n}$, with $e\in E^{y'}_n$ and $e'\in E^{x'}_{n-1}$.
We have $0=\pi_{\iota_{n-1}}d_0(v)=(-1)^{n-1}R^{g'}_1(e')+R_0^{y'}(e)$ and $0=\pi_{\delta_{n-1}}(v)=R_0^{x'}(e')$. It follows that $e'\in Z^{x'}_{n-1}$ is a cycle in the fiber and therefore $R^{g'}_1(e')\in Z^{y'}_{n-1}$. Then $\ker d_{n,n}^q$ identifies with the kernel of 
    $$Z^{x'}_{n-1}\oplus E^{y'}_n\to Z^{y'}_{n-1} \qquad (e',e)\mapsto (-1)^{n-1}R^{g'}_1(e')+R_0(e).$$
This is an epimorphism, for  $0\oplus E^{y'}_n$ covers $B^{y'}_{n-1}$, and since $R^{g'}_1$ is a quasi-isomorphism, $Z^{x'}_{n-1}\to Z^{y'}_{n-1}/B^{y'}_{n-1}=H^{y'}_{n-1}$ is surjective. 
    By Lemma \ref{lemma:cycles-borders} $\dim Z^{x'}_{n-1}=\dim Z^{y'}_{n-1}$ and we are done.

The descriptions we gave for the kernels $\ker d^q_{n,k}$, $k<n$, show that $C^{can}_n=\ker \pi_{\iota_n}$ is indeed a complement for $\ker d_{n,k}$, namely a cleavage. And is normal and weakly flat by Remark \ref{rmk:alt-formulas}.
\end{proof}

\begin{example} \label{example:sdp2}
Let $G$ be a Lie groupoid. We discuss now the compatibility of faces and degeneracies with the levelwise isomorphisms defined in Example \ref{example:sdp1}:
\begin{enumerate}[a)]
    \item When $G=M$ is a unit groupoid, the isomorphisms preserve the positive faces and degeneracies, and the 0-face up to a sign. Thus we get simplicial vector bundle isomorphisms
    $$(M\ltimes_RE)\cong DK(E,(-1)^{n-1}\partial)\cong DK(E,\partial).$$
    The second isomorphism is induced by $\epsilon_n:E_n\to E_n$, $\epsilon_n(e)=(-1)^{\binom{n}{2}}e$. 
    Note that when defining $d_0$, the case $\alpha=\beta'\delta_n$ is of type II in Dold-Kan and of type I in the semi-direct product, but they agree up to a sign because in this case $R_1=\id$;
    \item When $N=0$, the isomorphisms $(G\ltimes_RE)_n\cong Tr(R)_n$ commute with faces and degeneracies, hence yielding a higher vector bundle isomorphism
    $$(G\ltimes_R E)\cong Tr(R).$$
    Since they are nerves of groupoids, it is enough to show that source, target, multiplication and units are preserved, see Remark \ref{rmk:1-groupoid}. The $E$-valued cochains $(C(G,E),D)$ identify with the linear cochains in the semi-direct product $(C_{lin}(G\ltimes_RE),\delta)$;
    \item When $N=1$, the isomorphisms $(G\ltimes_RE)_n\cong Gr(R)_n$ commutes with every face and degeneracy. They are nerves of groupoids and we need to see that the structure maps are preserved. This is immediate for the source, target and unit. For the multiplication, given $v=(c_2,c_1,e,g)\in (G\ltimes_RE)_2$, with $e\in x_0^*E_0$ and $c_i\in x_i^*E_1$, we have
    \begin{align*}
    d_2(v)&=(c_1,e,g_1) \\
    d_1(v)&=(c_2,e,g_2g_1) \\
    d_0(v)&=(c_2-R^{g_2}_1(c_1)-R_2^{g}(e),R^{g_1}_1(e)+R^{x_1}_0(c_1),g_2)  
    \end{align*}
    Writing $d_1(v)=d_0(v)\cdot d_2(v)$, and comparing with the multiplication of the Grothendieck construction $Gr(R)$ recalled in Remark \ref{rmk:grothendieck}, we get a higher vector bundle isomorphism
    $$(G\ltimes_RE)\cong Gr(R).$$
\end{enumerate}
\end{example}



\begin{remark}\label{rmk:converse-thm1}
We have seen that the axioms RH1) and RH2) imply the simplicial identities of the semi-direct product. Conversely, let $G$ be a higher Lie groupoid,  $E=\bigoplus_{n=0}^N E_n$ a graded vector bundle over $G_0$, and  $R_m\in\Gamma(G_m;\hom^{m-1}(s_0^*E,t_0^*E))$ a sequence of arbitrary operators. Define $G\ltimes E$, $d_i$ and $u_j$ as before. Every simplicial identity other than $d_0u_0=\id$ and $d_0d_0=d_0d_1$ hold for arbitrary $R$. Our proof of Proposition \ref{prop:simplicial-identities} actually shows that:
\begin{enumerate}[a)]
    \item $d_0u_0=\id$ if and only if $R_1^{u(x)}|_{E_0}=\id$ and $R_{l+1}^{u_0(g)}|_{E_0}=0$; and
    \item $d_0d_0=d_0d_1$ if and only if RH2) holds.
\end{enumerate}
We cannot recover the axiom RH1) out of the simplicial identities. 
This axiom will follow from the fact that the cleavage $C^{can}$ must be normal and weakly flat, see Theorem \ref{thm:splitting}.
\end{remark}




After showing that the semi-direct product yields a well-defined higher vector bundle $(G\ltimes_RE)\to G$ endowed with a distinguished weakly flat cleavage, we study now the functoriality.

\begin{proposition}\label{prop:functoriality}
A morphism of representations up to homotopy  $\psi:(R:G\action E)\to (R':G\action E')$ induces a weakly flat morphism of higher vector bundles by the formula
$$\psi^\wedge:(G\ltimes_RE)\to(G\ltimes_{R'}E') \qquad
\pi_\beta \psi^\wedge = \sum_{r=0}^l  \psi^{g\beta\tau_{l-r}}_{l-r}\pi_{\beta\sigma_r}$$
The construction is functorial. 
Moreover, $\psi^\wedge$ is flat if and only if $\psi$ is a strict morphism, and $\psi^\wedge$ is the identity on the core if and only if $\psi$ is a gauge equivalence.
\end{proposition}


\begin{proof}
Let us check that $\psi^\wedge$ preserves faces and degeneracies. Given $\theta:[m]\to[n]$ preserving 0, 
$$\pi_\beta\theta^*\psi^\wedge=\pi_{\theta\beta}\psi^\wedge=
\sum_{r=0}^l  \psi^{g\theta\beta\tau_{l-r}}_{l-r}\pi_{\theta\beta\sigma_r}=
\sum_{r=0}^l  \psi^{(g\theta)\beta\tau_{l-r}}_{l-r}\pi_{\beta\sigma_r}\theta^*=\pi_\beta\psi^\wedge\theta^*$$
We are using the convention that $\pi_\gamma=0$ whenever $\gamma$ is not injective. Note that if $\theta\beta$ is not injective, either $\theta\beta\sigma_r$ is not injective and $\pi_{\theta\beta\sigma_r}=0$, or ${g\theta\beta\tau_{l-r}}$ is degenerate and $\psi^{g\theta\beta\tau_{l-r}}=0$ by axiom RH3). This shows that $\psi^\wedge$ preserves the positive faces and the degeneracies.

We check now that $\pi_\beta d_0\psi^\wedge=\pi_\beta\psi^\wedge d_0$.
If $(e,\alpha,g)\in (G\ltimes_RE)_n$ is a homogeneous element, then $\pi_\beta \psi^\wedge(e,\alpha,g)=\psi^{g\beta\tau_{l-k}}_{l-k}(e)$ if $\beta\sigma_k=\alpha$ and 0 otherwise. 
Let $(e,\alpha,g)\in (G\ltimes_RE)_n$. Suppose first that $\beta'\sigma_k=\alpha$. Then we have
\begin{align*}
\pi_\beta d_0 \psi^\wedge (e,\alpha,g) 
&= \sum_{r=0}^{l+1}(-1)^l {R'}^{g\beta'\tau_{l+1-r}}_{l+1-r}\pi_{\beta'\sigma_r}\psi^\wedge(e,\alpha,g) - \sum_{i=1}^l (-1)^i\pi_{\beta'\delta_i}\psi^\wedge(e,\alpha,g)\\
&= \sum_{r=k}^{l+1}(-1)^l {R'}^{g\beta'\tau_{l+1-r}}_{l+1-r}\psi_{r-k}^{g\beta'\sigma_r\tau_{r-k}}(e) - \sum_{i=k+1}^l (-1)^i
\psi_{l-k}^{g\beta'\delta_i\tau_{l-k}}(e)\\
&=(-1)^{k+1} \left( \sum_{\tilde r=0}^{\tilde m}(-1)^{\tilde m} {R'}^{\tilde g\tau_{\tilde r}}_{\tilde m-\tilde r}\psi_{\tilde r}^{\tilde g\sigma_{\tilde r}}(e) - \sum_{\tilde i=0}^{\tilde m-1}(-1)^{\tilde i}
\psi_{\tilde m-1}^{\tilde g\delta_{\tilde i}}(e) \right)
\end{align*}
where we are writing $\tilde r=r-k$, $\tilde m=l+1-k$, $\tilde i=i-k$ and $\tilde g=g\beta'\tau_{l+1-k}$. 
And also,
\begin{align*}
\pi_\beta \psi^\wedge d_0 (e,\alpha,g) 
&= \sum_{r=0}^{l}\psi_{l-r}^{(g\delta_0)\beta\tau_{l-r}} \pi_{\beta\sigma_r} d_0 (e,\alpha,g) \\
&= \sum_{r=k-1}^{l} \psi_{l-r}^{(g\delta_0)\beta\tau_{l-r}} 
(-1)^r  R^{g(\beta\sigma_r)'\tau_{r+1-k}}_{r+1-k}(e)\\
&= (-1)^{k+1} \left( 
\sum_{\tilde r=0}^{\tilde m}  (-1)^{\tilde r} \psi_{\tilde m-\tilde r}^{\tilde g\tau_{\tilde m-\tilde r}} 
 R^{\tilde g\sigma_{\tilde r}}_{\tilde r}(e) \right)
\end{align*}
where now $\tilde r=r+1-k$, $\tilde m=l+1-k$ and $\tilde g=g\beta'\tau_{l+1-k}$. The identity follows by axiom RH4). The fact that $\psi^\wedge$ is weakly flat is immediate from the definition: if $s_0(v)=0$ and $s_k(v)\in C$ for every $k$, then $\pi_{\sigma_r}(v)=0$ for every $r$, and therefore $\pi_{\iota}\psi^\wedge(v)=0$ and $\psi^\wedge(v)\in C'$.

The proofs that $\psi\mapsto\psi^\wedge$ is functorial, that $\psi^\wedge$ is flat if and only if $\psi$ is strict, and that $\psi^\wedge$ is the identity on the core if and only $\psi$ is a gauge equivalent, are straightforward.
\end{proof}


\begin{remark}
Summarizing, we can regard the semi-direct product as a functor in two alternative ways, either by considering the canonical cleavage or by disregarding it: $$\ltimes_+:Rep^\infty(G)\to VB^\infty_+(G) \qquad 
\ltimes:Rep^\infty(G)\to VB^\infty(G)$$
They are related by the forgetful functor $U$, namely $\ltimes =U\circ\ltimes_+$.
The semi-direct product preserves the order, so we have well-defined restrictions $\ltimes^N_+:Rep^N(G)\to VB^N_+(G)$ and $\ltimes^N:Rep^N(G)\to VB^N(G)$, related by the forgetful functor $U^N$. 
\end{remark}


\section{Splitting higher vector bundles}
\label{section:splitting}

Here we discuss the push-forwards in the smooth linear setting, show that a higher vector bundle endowed with a suitable cleavage splits into a representation up to homotopy, and complete the proof of our Main Theorem.

\


Let $q: V\to G$ be a higher vector bundle with core $E=\bigoplus_{n=0}^N E_n$ and let $C$ be a cleavage.
We write $K_n=\ker (d_{n,0}:V_n\to V_{n,0})$, so $E_n=K_n|_{G_0}$ and $V_n=K_n\oplus C_n$, and $\pi_K:V_n\to K_n$, $v\mapsto v_K$, is the projection with kernel $C_n$. 
We identify injective maps $[k]\xto\alpha[n]$ with sub-simplicial sets $\alpha\subset\Delta[n]$ and write $\alpha^*(v)=v|_\alpha$. Remember that we use $\sigma_k:[k]\to[n]$ for the inclusion, and $\delta_i:[n-1]\to[n]$ and $\upsilon_j:[n+1]\to[n]$ for the elementary injections and surjections. 
Let us discuss the push-forward of Lemma \ref{lemma:push-forward} in the smooth linear setting.

\begin{proposition}\label{prop:smooth-push-forward}
Given $i<n$, the push-forward $p_i:V_n\to V_n$ is a smooth vector bundle map covering $d_iu_i:G_n\to G_n$. If $v\in V_n$ and $\alpha\subset[n]$ are such that $v|_\alpha=0$ then $p_i(v)|_\alpha=0$. The push-forward gives a fiber-wise isomorphism $p_i: K_n\to K_n$. Moreover:
\begin{enumerate}[1)]
    \item If $C$ is normal and $g=q(v)\in u_{i+1}(G_{n-1})$ then $p_i(v)=v$;
    \item If $C$ is weakly flat, $v|_{\sigma_i}=0$ and $v|_\alpha\in C$ then $p_i(v)|_\alpha\in C$.
\end{enumerate}
\end{proposition}

\begin{proof}
The push-forward $p_i$ is smooth and linear because the cleavage $C$ is so. Suppose that $v|_\alpha=0$. If $i\notin \alpha$ then $p_i(v)|_\alpha=v|_\alpha=0$, and if $i\in \alpha$ then, by the inductive construction, $h_i(v)|_{\delta_{i+1}\alpha\cup\{i+1\}}=0$, and therefore, $p_i(v)|_\alpha=0$, where $h_i(v)\in V_{n+1}$ is as in Lemma \ref{lemma:push-forward}.

Clearly $p_i(K_n)\subset K_n$, for  $K_n=\{v:v|_{\delta_j}=0\ \forall j>0\}$. We claim that $p_i: K_n\to K_n$ is fiber-wise injective. Let $v\in K_n$ such that $p_i(v)=0$. 
To show that $v=d_{i+1}h_i(v)=0$, since $h_i(v)\in C_{n+1}$, it is enough to show that every other face of $h_i(v)$ vanishes.  
First, $d_i(h_i(v))=p_i(v)=0$ by assumption. 
Then, by the inductive construction of $h_i(v)$, we have $d_jh_i(v)=0$ for $j\notin\{0,i,i+1\}$. 
If $i=0$ we are done. If $i\neq 0$, since $d_0h_i(v)=h_{i-1}d_0(v)$ and $p_{i-1}d_0(v)=d_0p_i(v)=0$, by an inductive argument on $i$ we have $d_0(v)=0$, hence $d_0h_i(v)=0$, and we are done.

Claim 1) easily follows from Lemma \ref{lemma:push-forward}, part 1). Regarding 2), recall that weakly flat means flat over the zero section, see Definition \ref{def:linear-cleavage}. If $i\notin\alpha$ then $p_i(v)|_\alpha=v|_\alpha\in C$, and if $i\in\alpha$, the hypothesis of Lemma \ref{lemma:push-forward}, part 2), hold and the result follows.
\end{proof}

\begin{lemma}\label{lemma:commuting-pf-smooth}
If $v|_{\sigma_r}=0$ and $i\leq j\leq r$ then $$p_ip_j(x)=
\begin{cases}
p_{j}p_i(x) &j>i+1\\
p_{i}p_{i+1}p_{i}(x) & j=i+1\\
p_i(x) & j=i 
\end{cases}$$
In particular, if $v\in K_n$, the identities hold for every $i,j$.
\end{lemma}

\begin{proof}
It immediately follows from Proposition \ref{prop:commuting-push-forward}.
\end{proof}

%


We define the {\bf retraction} $r:V_n\to V_n\r{G_0}$ as the vector bundle map covering $x_n:G_n\to G_0\subset G_n$ given by pushing forward each vertex to the last fiber in the following order:
$$r=(p_0p_1p_2\dots p_{n-1})\dots(p_0p_1p_2)(p_0p_1)(p_0)$$
The retraction restricts to a fiber-wise isomorphism
$r:K_n\to K_n|_{G_0}=E_n$ by Proposition \ref{prop:smooth-push-forward}, and if $C$ is normal then $r|_{G_0}:E_n\to E_n$ is just the identity.
By Lemma \ref{lemma:push-forward}, we have $rd_0=d_0r$.
In general $r(C)\not\subset C$. Down below we depict how the retraction $r_2:V_2\to V_2|_{G_0}$ works:
$$
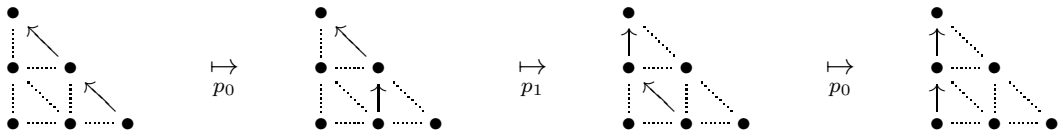
\begin{matrix}\xymatrix@R=10pt@C=10pt{
 \bullet & &  \\
 \bullet \ar@{.}[u] &  \bullet \ar[ul] \ar@{.}[l]&   \\
 \bullet \ar@{.}[u] & \bullet \ar@{.}[u] \ar@{.}[ul] \ar@{.}[l]& \bullet \ar[ul] \ar@{.}[l]  
}\end{matrix} \qquad \underset{p_0}\mapsto \qquad
\begin{matrix}\xymatrix@R=10pt@C=10pt{
 \bullet & &  \\
 \bullet \ar@{.}[u] &  \bullet \ar[ul] \ar@{.}[l]&   \\
 \bullet \ar@{.}[u] & \bullet \ar[u] \ar@{.}[ul] \ar@{.}[l]& \bullet \ar@{.}[ul] \ar@{.}[l]  
}\end{matrix} \qquad \underset{p_1}\mapsto \qquad
\begin{matrix}\xymatrix@R=10pt@C=10pt{
 \bullet & &  \\
 \bullet \ar[u] &  \bullet \ar@{.}[ul] \ar@{.}[l]&   \\
 \bullet \ar@{.}[u] & \bullet \ar@{.}[u] \ar[ul] \ar@{.}[l]& \bullet \ar@{.}[ul] \ar@{.}[l]  
}\end{matrix} \qquad \underset{p_0}\mapsto \qquad
\begin{matrix}\xymatrix@R=10pt@C=10pt{
 \bullet & &  \\
 \bullet \ar[u] &  \bullet \ar@{.}[ul] \ar@{.}[l]&   \\
 \bullet \ar[u] & \bullet \ar@{.}[u] \ar@{.}[ul] \ar@{.}[l]& \bullet \ar@{.}[ul] \ar@{.}[l]  
}\end{matrix} 
$$

We want to show that if $q: V\to G$ admits a normal weakly flat cleavage then $V$ is a semi-direct product. 
Let $W_n=\bigoplus_{[k]\xto\alpha[n]}x_{\alpha(k)}^*E_k$, where the sum is over the injective maps preserving 0, and let $\phi_n: V_n\to W_n$ be the vector bundle map over $G_n$ defined by $\pi_\alpha\phi_n(v)= r \pi_K (v|_\alpha)$. 
$$\xymatrix@R=10pt{
V_n \ar[d]_{\phi_n} \ar[r]^a  & V_k \ar[d]^{\pi_K} & \\
W_n \ar[rdr]_{\pi_\alpha}  & K_k \ar[rd]^{r} & \\
& & E_k \\
G_n \ar[r]^{a} & G_k \ar[r]^{x_k} & G_0}$$
By definition, $v\in C_n$ if and only if $\pi_{\iota_n}\phi_n(v)=0$, and more generally, 
$v|_\alpha\in C_k$ if and only if $\pi_{\alpha}\phi_n(v)=0$.
If $v\in K_n$ then $\pi_\alpha\phi_n(v)=0$ for $\alpha\neq\iota_n$, and $\pi_{\iota_n}\phi_n(v)=r(v)$. 

\begin{lemma}\label{lemma:direct-sum}
The map $\phi_n:V_n\to W_n$ is a vector bundle isomorphism over $G_n$.
\end{lemma}

\begin{proof}
By Dold-Kan we have $V_n|_{G_0}\cong W_n|_{G_0}$, see Remark \ref{rmk:dk-projective}, so $V_n$ and $W_n$ have the same rank. We will show that $\phi_n$ is fiber-wise injective.
Suppose that $v\neq0$. Let $[k]\xto\alpha[n]$, $\alpha(0)=0$, such that $v|_\alpha\neq0$, with $k$ minimal. Then $v|_\alpha\in K_k$, for $d_i(v|_\alpha)=v|_{\alpha\delta_i}=0$.
Since $r|_K$ is fiber-wise isomorphic, we have 
$0\neq r(v|_\alpha)=r\pi_K(v|_\alpha)=\pi_\alpha\phi_n(v)$ and the result is proven.
\end{proof}

We translate the simplicial structure from $V$ to $W$ via $\phi$, setting $d_i=\phi d_i\phi^{-1}:W_n\to W_{n-1}$ and $u_j=\phi u_j\phi^{-1}:W_n\to W_{n+1}$. The next lemma shows that the positive faces and the degeneracies behave exactly as if $V$ were a semi-direct product, see Definition \ref{def:sdp-faces-degeneracies}.


\begin{lemma}\label{lemma:splitting-positive-faces-deg}
We have $\pi_\beta  d_i = \pi_{\delta_i\beta}$ and if $C$ is normal then $\pi_\beta u_j=\pi_{\upsilon_j\beta}$. In particular,
$$d_i(e,\alpha,g)=\begin{cases}
(e,\upsilon_i\alpha,g\delta_i) & i\notin\alpha\\ 
0 & i\in\alpha \end{cases}\qquad
u_j(e,\alpha,g)=\begin{cases}
(e,\delta_{1}\alpha,g\upsilon_0) & j=0\\
(e,\delta_j\alpha,g\upsilon_j)& j\notin\alpha\\
(e,\delta_j\alpha,g\upsilon_j)+(e,\delta_{j+1}\alpha,g\upsilon_j)& 0\neq j\in\alpha
\end{cases}$$
\end{lemma}

\begin{proof}
Regarding the positive faces, given $\beta:[l]\to[n-1]$, 
the definition of $\phi$ readily implies that
$\pi_\beta d_i=(\pi_\beta\phi) d_i\phi^{-1}=r\pi_K \beta^* d_i \phi^{-1}
=(\pi_{\delta_i\beta}\phi)\phi^{-1}=\pi_{\delta_i\beta}$.

Regarding the degeneracies, given $\beta:[l]\to[n+1]$, we consider two cases. 
If $\upsilon_j\beta$ is injective then $\pi_\beta u_j=(\pi_\beta\phi) u_j\phi^{-1}= r\pi_K b u_j\phi^{-1} = (\pi_{\upsilon_j\beta}\phi)\phi^{-1}=\pi_{\upsilon_j\beta}$ by the definition of $\phi$. If $\upsilon_j\beta$ is not injective then $\pi_{\upsilon_j\beta}=0$ by convention, we can write $\upsilon_j\beta=\gamma\upsilon_{j'}$ for some $\gamma$ and $j'$, and $\pi_\beta u_j = r \pi_K b u_j\phi^{-1} = r \pi_K u_{j'} \gamma^*\phi^{-1}=0$, for $\pi_K$ vanishes over the cleavage $C\supset u_{j'}(V_{n-1})$.

The descriptions for homogeneous vectors are consequences of the proven identities.
\end{proof}    


We want to describe $d_0:W_n\to W_{n-1}$. 
Let us first establish some technical auxiliary results. 
We define the {\bf intermediate retraction} $r_\lambda:V_n\to V_n$ as the linear map given by
$r_\lambda=p_0\dots p_{\lambda-1}$. 
Since the push-forward $p_i$ covers the map $d_iu_{i+1}:G_n\to G_n$, the intermediate retraction 
$r_\lambda$ covers $d_0u_1\dots d_{\lambda-1}u_{\lambda}=u_0^\lambda d_0^\lambda$. 
$$\begin{tikzpicture}[scale=.4]
\draw[thick] (0,5) -- (5,0) ;
\draw[dashed] (0,0) -- (0,5) ;
\draw[dashed] (3,0) -- (3,2) -- (0,5) ;
\draw (0,0) node[below] {\scriptsize $n$};
\draw (3,0) node[below] {\scriptsize $\lambda$};
\draw (5,0) node[below] {\scriptsize $0$};
\draw (3.5,1.5) node[right] {\scriptsize $x$};
\end{tikzpicture} \qquad \qquad
\begin{tikzpicture}[scale=.4]
\draw[dashed] (0,5) -- (5,0) ;
\draw[dashed] (0,0) -- (0,5) ;
\draw[thick] (3,0) -- (3,2) -- (0,5) ;
\draw (0,0) node[below] {\scriptsize $n$};
\draw (3,0) node[below] {\scriptsize $\lambda$};
\draw (5,0) node[below] {\scriptsize $0$};
\draw (3,1.5) node[left] {\scriptsize $r_\lambda(x)$};
\end{tikzpicture} \qquad\qquad
\begin{tikzpicture}[scale=.4]
\draw[dashed] (0,5) -- (5,0) ;
\draw[thick] (0,0) -- (0,5) ;
\draw[dashed] (3,0) -- (3,2) -- (0,5) ;
\draw (0,0) node[below] {\scriptsize $n$};
\draw (3,0) node[below] {\scriptsize $\lambda$};
\draw (5,0) node[below] {\scriptsize $0$};
\draw (0,1.5) node[left] {\scriptsize $r(x)$};
\end{tikzpicture}$$
As we did with the faces and degeneracies, we translate the push-forward $p_i$, the retraction $r$, and the intermediate retractions $r_\lambda$ to $W$ via $\phi$, as 
$p_i=\phi p_i\phi^{-1}$, $r=\phi r\phi^{-1}$ and
$r_\lambda=\phi r_\lambda\phi^{-1}$.

\begin{lemma}\label{lemma:first-arrows}
If $C$ is normal and weakly flat, then: 
\begin{enumerate}[i)]
    \item $p_i(e,\alpha,g)=(e,\alpha,g\upsilon_{i+1}\delta_i)$ if $i<\alpha(k)$.
    \item $r_{\alpha(k)}(e,\alpha,g)=(e,\alpha,g\delta_0^{\alpha(k)}\upsilon_0^{\alpha(k)})$.
\end{enumerate}
\end{lemma}

\begin{proof}
Let us prove ii).
Given $i<\alpha(k)$, write $g'=g\delta_{i+1}^{\alpha(k)-i-1}\upsilon_{i+1}^{\alpha(k)-i-1}$.
By an inductive argument, we need to show that $p_i(e,\alpha,g')=(e,\alpha,g'\upsilon_{i+1}\delta_i)$.
By Lemma \ref{lemma:splitting-positive-faces-deg}, we have that
$(e,\alpha,g')|_{\sigma_i}=d_{i+1}\dots d_n(e,\alpha,g')=0$. Given $\beta\neq\alpha$, since $(e,\alpha,g')|_\beta\in C$, by Proposition \ref{prop:smooth-push-forward}, $p_i(e,\alpha,g')|_\beta\in C$, and $p_i(e,\alpha,g')=(e',\alpha,g'\upsilon_{i+1}\delta_i)$ for some $e'$. 
It remains to show that $e'=e$.

By definition, $e'=
\pi_\alpha p_i(e,\alpha,g')=
r\pi_K((p_i(e,\alpha,g'))|_\alpha)$. We claim that $(p_i(e,\alpha,g'))|_\alpha=p_{i'}((e,\alpha,g')|_{\alpha})$ for some $i'$. This is because $g'|_{\{\alpha(k),\dots,i+1\}}$ is totally degenerate, see also Remark \ref{rmk:push-forward-faces}.
Then we have
$e'=\pi_Kp_{i'}((e,\alpha,g')|_{\alpha})$. Since
$(e,\alpha,g')|_{\alpha}\in K_k$, 
by Proposition \ref{prop:smooth-push-forward}, 
$p_{i'}((e,\alpha,g')|_{\alpha})\in K_k$ as well, so we can disregard the projection $\pi_K$ over $K_k$, and get $e'=rp_{i'}((e,\alpha,g')|_{\alpha})$.
Finally, by Lemma \ref{lemma:commuting-pf-smooth}, we can rearrange the pushforwards appearing in $rp_{i'}$ in increasing order to conclude $rp_{i'}|_{K_k}=r|_{K_k}$ and $e'=r((e,\alpha,g')|_{\alpha})=e$.

We now prove i) using ii). By the above arguments, we know that $p_i(e,\alpha,g)=(e',\alpha,g\upsilon_{i+1}\delta_i)$ for 
$e'=\pi_\alpha r_{\alpha(k)}p_i(e,\alpha,g)$. By
Lemma \ref{lemma:commuting-pf-smooth}, since every pushforward appearing in $r_{\alpha(k)}p_i$ is of the form $p_j$, $j<\alpha(k)$, we can rearrange them to get $r_{\alpha(k)}p_i=r_{\alpha(k)}$, and conclude that $e'=e$.
\end{proof}


The next lemma shows that $d_0$, when expressed under the direct sum given by Lemma \ref{lemma:direct-sum}, increases the support, and also that the components $\pi_{\iota_n} d_0$ determine all the others.

\begin{lemma}\label{lemma:support}
If $i\notin\beta'$ then $\pi_\beta d_0(e,\alpha,g)=
\begin{cases}
 \pi_{\upsilon_{i-1}\beta} d_0(e,\upsilon_i\alpha,g\delta_i) & i\notin\alpha\\
0 & i\in\alpha
\end{cases}$
\end{lemma}

\begin{proof}
Since $i\notin\beta'$ we can write $\beta=\delta_{i-1}\upsilon_{i-1}\beta$. Then 
$$\pi_\beta d_0(e,\alpha,g)=\pi_{\delta_{i-1}\upsilon_{i-1}\beta}  d_0(e,\alpha,g)=\pi_{\upsilon_{i}\beta}   d_{i-1}d_0(e,\alpha,g)=\pi_{\upsilon_{i}\beta}  d_0d_{i}(e,\alpha,g)$$
and by Lemma \ref{lemma:splitting-positive-faces-deg}, if $i\in\alpha$ then $d_{i}(e,\alpha,g)=0$ and
if $i\notin\alpha$ then $d_{i}(e,\alpha,g)=(e,\upsilon_i\alpha,g\delta_i)$.
\end{proof}


We have everything in place to prove our second main theorem, namely that we can split a higher vector bundle using a normal weakly flat cleavage.



\begin{theorem}\label{thm:splitting}
Given $q: V\to G$ be a higher vector bundle with core $E=\oplus_{n=0}^NE_n$, and $C$ a normal weakly flat cleavage, the operators $R_m\in\Gamma\big(G_m;\Hom^{(m-1)}(s_0^*E,t_0^*E)\big)$ defined by 
$$R_m^g(e)=(-1)^{m+n-1}
\pi_{\iota}d_0(e,\sigma_n,u_0^n(g)) \qquad 
g\in G_m\quad  e\in E_n^{g\sigma_0}$$ 
give a well-defined representation up to homotopy $R: G\action E$. Moreover, $W=G\ltimes_R E$ is the semi-direct product of $R$, and $\phi: V\to G\ltimes_R E$ is an isomorphism satisfying $\phi(C)=C^{can}$.
\end{theorem}

If $V=(G\ltimes_RE)$ is a semi-direct product and $C^{can}$ is the canonical cleavage, then the above formulas recover the operators $R$.

\begin{proof}
Let us show that $d_0=\phi d_0\phi^{-1}:W_n\to W_{n-1}$ satisfy the Definition \ref{def:sdp-faces-degeneracies}.
By Lemma \ref{lemma:support}, writing $\iota=\iota_{n-1}$, it is enough to show the following formulas:
$$\pi_{\iota} d_0(e,\alpha,g)=\begin{cases}
(-1)^{n-1} R^{g\tau_{n-k}}_{n-k}(e) & \alpha=\sigma_k \\
(-1)^{i+1}e & \alpha=\delta_i,\ 0<i<l+1\\
0 & \text{otherwise}.
\end{cases}$$
By Lemma \ref{lemma:support}, since $d_0$ increases the support, we can write $d_0(e,\alpha,g)=\sum_{\beta'\supset\alpha}(e_\beta,\beta,g\delta_0)$. 
We claim that it is enough to show the formulas in the case $g=g\delta_0^{\alpha(k)}\upsilon_0^{\alpha(k)}$.
In fact, we have
\begin{align*}
\pi_{\iota}   d_0(e,\alpha,g\delta_0^{\alpha(k)}\upsilon_0^{\alpha(k)})&=
\pi_{\iota}   d_0   r_{\alpha(k)}(e,\alpha,g) && \text{by Lemma \ref{lemma:first-arrows}}\\
&= 
\pi_{\iota}   r_{{\alpha(k)}-1}   d_0(e,\alpha,g) && \text{by Remark \ref{rmk:push-forward-faces}}\\
&=
\pi_{\iota}   r_{{\alpha(k)}-1}\sum_{\beta'\supset\alpha}(e_\beta,\beta,g\delta_0) && \\
&=\pi_{\iota} \sum_{\beta'\supset\alpha}(e_\beta,\beta,g\delta_0^{{\alpha(k)}} \upsilon_0^{{\alpha(k)}-1})
 && \text{by Lemma \ref{lemma:first-arrows}}\\
&= e_{\iota}=\pi_{\iota}   d_0(e,\alpha,g)
\end{align*}
Suppose then that 
$g=g\delta_0^{\alpha(k)}\upsilon_0^{\alpha(k)}$
and consider the following situations:
\begin{itemize}
    \item If $\alpha=\sigma_k$, $0\leq k\leq n$, then $\pi_\iota d_0(e,\sigma_k,g)=(-1)^{n-1}R_{n-k}^{g\tau_{n-k}}(e)$ because $g=g\delta_0^{\alpha(k)}\upsilon_0^{\alpha(k)}$, and because of the way we have defined $R$;
    \item If $\alpha=\delta_{i}$, $0<i<n$, then $k=n-1$, $\alpha(k)=n$, and $g=u_0^{n}(x)$ is totally degenerated, $x=s_0(g)\in E_0$. Then we can restrict our attention to the fiber $E^x$. By the Lemma \ref{lemma:dk-groupoid}, $C^x=D^x$ is spanned by the degenerate simplices, and $\pi_\iota d_0(e,\delta_i,u_0^n(x))=(-1)^{i+1}e$ by our formulas for Dold-Kan in Proposition \ref{prop:inverse-dk};
    \item If $\alpha\neq\sigma_k,\delta_i$, then $s_0(e,\alpha,g)=0$, $s_k(e,\alpha,g)\in C$ for every $k$, and $d_i(e,\alpha,g)\in C$ for every $i>0$. Since $C$ is weakly flat, $d_0(e,\alpha,g)\in C$, or equivalently, $\pi_\iota d_0(e,\alpha,g)=0$. 
\end{itemize}


We have shown that the faces and degeneracies of $V$, under the isomorphisms $\phi_n$, correspond exactly to the formulas of the semi-direct product in Definition \ref{def:sdp-faces-degeneracies}. It remains to prove that the operators $R_n$ satisfy the axioms of a representation up to homotopy. As observed in Remark \ref{rmk:converse-thm1}, the simplicial identities of $V$ imply that 
$R$ satisfies RH2) and a part of RH1), namely $R_1^{u(x)}|_{E_0}=\id$ and $R_{m}^{u_0(g)}|_{E_0}=0$ for $m\geq 2$. To complete RH1) we still need to check:
\begin{enumerate}[i)]
    \item $R_1^{u(x)}|_{E_n}=\id$ for arbitrary $n$; and 
    \item $R_{m}^{u_i(g)}|_{E_n}=0$ for $n>0$, $m\geq 2$ and arbitrary $i$.
\end{enumerate}    
Condition i) follows by restricting to the fiber $V^x$ and using Proposition \ref{prop:inverse-dk}. To prove ii), given $e\in E_n$, $n>0$, and $g\in G_{m-1}$, we must check that $d_0(e,\sigma_n,u_0^n u_i(g))\in C$. We use Lemma \ref{lemma:splitting-positive-faces-deg}.
If $i>0$ then
$d_0(e,\sigma_n,u_0^n u_i(g))=d_0u_{i+n}(e,\sigma_n,u_0^n(g))=u_{i+n-1}d_0(e,\sigma_n,u_0^n(g))$ is in $C$ for it is normal. 
If $i=0$, we can write $d_0(e,\sigma_n,u_0^n u_0(g))=u_{n-1}d_0(e,\sigma_n,u_0^n(g))-d_0(e,\delta_{n}\sigma_n,u_nu_0^n(g))$,
the first term is degenerate, and the second is in $C$, for $\delta_n\sigma_n$ is neither $\sigma_k$ nor $\delta_i$. 
\end{proof}


Theorems \ref{thm:sdp} and \ref{thm:splitting} combine to give a weak version of Theorem \ref{thm:main-intro}. 

\begin{corollary}
The semi-direct product sets a 1-1 correspondence between representations up to homotopy $R: G\action E$ (modulo strict isomorphisms) and higher vector bundles $q: V\to G$ coupled with a normal weakly flat cleavage $C$ (modulo flat isomorphisms).
\end{corollary}

Our final goal is to upgrade this correspondence to an equivalence of categories. Note that when splitting a higher vector bundle $q:V\to G$, while $R_0$ is canonical, the higher operators $R_n$ may depend on the choice of $C$. 
The dependence of $R$ on the cleavage $C$ can be framed within the problem of studying when a morphism $\phi:(G\ltimes E)\to (G\ltimes E')$ is in the image of $\ltimes$.

\begin{proposition}\label{prop:vee}
Let $\phi:(G\ltimes_ E)\to (G\ltimes_{R'} E')$ be a higher vector bundle morphism between semi-direct products, equipped with their canonical cleavages. 
\begin{enumerate}[a)]
    \item If $i\notin\beta$, then $\pi_\beta \phi(e,\alpha,g)=\begin{cases}\pi_{\upsilon_{i}\beta} \phi(e,\upsilon_i\alpha,g\delta_i) & i\notin\alpha\\ 0 & \text{otherwise} \end{cases}$
    \item If $\phi$ is weakly flat and $(e,\alpha,g)\in G\ltimes_RE$ then $r_{\alpha(k)}\phi(e,\alpha,g)=\phi r_{\alpha(k)}(e,\alpha,g)$. 
    \item If $\phi$ is weakly flat then $\phi^\vee:(R:G\action E)\to (R':G\action E')$ given by 
$$(\phi^\vee)_m\in\Gamma\big(G_m;\Hom^{(m)}(s_0^*E,t_0^*E')\big)
\qquad (\phi^\vee)_m^g(e)=\pi_{\iota}\phi(e,\sigma_n,u_0^n(g)),$$
is a morphism of representations up to homotopy such that $(\phi^\vee)_0=\phi|_{E}:E\to E'$.
    \item  $(\psi^\wedge)^\vee=\psi$ for every $\psi$, and $(\phi^\vee)^\wedge=\phi$ when $\phi$ is weakly flat.
\end{enumerate}
\end{proposition}

\begin{proof}
Item a) asserts that $\phi$ increases the support and that the components $\pi_\iota\phi$ determine every other $\pi_\beta\phi$. Its proof is analogous to that of Lemma \ref{lemma:support}, replacing $d_0$ by $\phi$.

To prove b), it is enough to show that $\phi$ commutes with the first push-forwards, and this follows from the identity 
$\phi h_i(e,\alpha,g)=h_i\phi(e,\alpha,g)$ if $i<\alpha(k)$. Let us prove this, namely that $\phi h_i(e,\alpha,g)$ is the $i$-th transport of $\phi(e,\alpha,g)$. 
Items (i) and (ii) of Lemma \ref{lemma:push-forward} are clear. To prove (iii),
we have to check that $(\phi h_i(e,\alpha,g))|_\beta\in C$ for every $\{i+1,i\}\subset\beta\subset[n]$, and since $\phi$ is weakly flat, this would follow from $ h_i(e,\alpha,g)|_{\beta\sigma_r}\in C$ for every $r>0$ and $h_i(e,\alpha,g)|_{\beta\sigma_0}=0$.
Since $d_{i+1}h_i(e,\alpha,g)=(e,\alpha,g)$, it follows from the formulas in Remark \ref{rmk:alt-formulas} that $h_i(e,\alpha,g)|_{\sigma_i}=0$, and therefore $h_i(e,\alpha,g)|_{\beta\sigma_r}=0$ whenever $\beta(r)\leq i$. And if $\beta(r)>i$, then $\{i,i+1\}\in \beta\sigma_r$, and $h_i(e,\alpha,g)|_{\beta\sigma_r}\in C$ by property (iii) of $h_i$.



Regarding c), if $m>0$, then 
$\phi(e,\sigma_n,u_0^nu_i(g))=u_{i+n}\phi(e,\sigma_n,u_0^n(g))$ is degenerate, belongs to $C$, and therefore $(\phi^\vee)_m^{u_i(g)}(e)=0$, proving  RH3).  The argument behind the proof of RH4) is the same as in Proposition \ref{prop:functoriality}.
Given $m\geq 0$, $g\in G_m$ and $e\in E_n^{s_0(g)}$, writing $\tilde g=g\upsilon^n_0$, and using the previous items a) and b), the formula for $\pi_\iota d_0$ in Definition \ref{def:sdp-faces-degeneracies} applied to $(-1)^{n-1}\pi_\iota d_0\phi(e,\sigma_n,\tilde g)$ yields the left-hand side of RH4) for $\phi^\vee$. Similarly, $(-1)^{n-1}\pi_\iota \phi d_0(e,\sigma_n,\tilde g)$ yields the right-hand side of RH4).
Since $d_0\phi=\phi d_0$, we get RH4). 

The claims in d) are clear.
\end{proof}


\begin{remark}
If $\psi:(R:G\action E)\to (R':G\action E)$ is a gauge equivalence then $\psi^\wedge:(G\ltimes_RE)\to(G\ltimes_{R'}E)$ is a weakly flat isomorphism by Proposition \ref{prop:functoriality}, and the {\bf twisting} $C^\psi=(\psi^\wedge)^{-1}(C^{can})$ of $C^{can}$ by $\psi$ is a new normal weakly flat cleavage in $G\ltimes_R E$ such that $\id$ is weakly flat:
$$\id:(G\ltimes_RE,C^{can})\to(G\ltimes_RE,C^\psi).$$
For VB-groupoids, every normal cleavage in $G\ltimes E$ is the twisting of $C^{can}$ by a gauge equivalence, and the representations up to homotopy coming from different choices of normal cleavages form a gauge-equivalence class, see \cite[Thm 5.10]{gsm}. In general, if $C$ is normal and weakly flat and 
$\id:(G\ltimes_RE,C^{can})\to(G\ltimes_RE,C)$ is weakly flat, then $C=C^\psi$ for some $\psi$. Example \ref{ex:not-full} shows that there are normal weakly flat cleavages on $G\ltimes_RE$ that are not of the type $C^\psi$. 
\end{remark}


These results show a 1-1 correspondence between representations up to homotopy up to isomorphisms and higher vector bundles equipped with a cleavage up to weakly flat isomorphism. More generally, 
Propositions \ref{prop:functoriality} and \ref{prop:vee} combine to produce our Main Theorem \ref{thm:main-intro}



\begin{example}\

\begin{enumerate}[a)]
    \item When $G=M$ is a unit groupoid, any higher vector bundle admits a unique normal cleavage, see Example \ref{ex:cleavages}, which is flat, see Lemma \ref{lemma:dk-groupoid}, and any morphism is flat, so $U$ is an equivalence, and we recover the classic Dold-Kan correspondence, see Remark \ref{rmk:dk-projective}. 
    \item Combining Remark \ref{rmk:forgetful} and Theorem \ref{thm:main-intro} we recover the equivalence between representations and 0-vector bundles $\ltimes:Rep^0(G)\to VB^0(G)$ mentioned in Proposition \ref{prop:equiv-def}.
    \item When working with VB-groupoids, again by Remark \ref{rmk:forgetful},  we recover the equivalence given by the Grothendieck construction $\ltimes: Rep^1(G)\to VB^{1}(G)$ proven in \cite[Thm 6.1]{gsm} and \cite[Thm 2.7]{dho}. 
\end{enumerate}    
\end{example}

\begin{remark}\label{rmk:final}
Theorem \ref{thm:main-intro}, even when $G$ is a set-theoretic higher groupoid, provides a novel and significant result: a concrete version of the higher Grothendieck correspondence. Studies in this direction have been conducted by Lurie, among others. In \cite[\S 3.2]{lurie}, a general equivalence is established via the {\em unstraightening functor}. When $G = C$ is a small category, the unstraightening of a strict functor $f: C \to sSets$ is modeled by the {\em relative nerve} $N_fC$ \cite[\S 3.2.5]{lurie}, and a simpler, more direct approach is presented in \cite{hm}, using the {\em standard model} $h_!(f)$ for the homotopy colimit. From this perspective, our semi-direct product $G \ltimes_R E$ plays the role of the relative nerve and the standard model, providing a concrete model for the homotopy colimit of the non-strict functor $R: G \to Ch_{\geq 0}(Vect_\R)$ induced by the representation up to homotopy.
\end{remark}


\section{Future directions and open problems}

In this final section we list some work in progress, describe potential applications, and leave some open problems related to our paper. 

\subsection*{Existence of cleavages and the derived equivalence}

The Dold-Kan and the Grothendieck correspondences are equivalences of categories, but they also induce equivalences at the derived level. 
Similarly, a derived version of our main theorem will be addressed in \cite{dht}. Given a higher vector bundle $q: V\to G$, in order to define a representation up to homotopy, we have used a normal weakly flat cleavage, but we have seen in Example \ref{ex:es} that higher vector bundles may not admit a cleavage at all.  

We solve this issue by relaxing the notion of cleavage. Given $q:V\to G$ a higher vector bundle, define a {\bf pre-cleavage} $C_{n,k}$, $0\leq k<n$, as a collection of sub-bundles of $C_{n,k}\subset V_n$ such that the relative horn maps $d^q_{n,k}:C_{n,k}\to d_{n,k}^*V_{n,k}$ are isomorphism, and call a pre-cleavage {\bf normal} if each $C_{n,k}$ contains the degenerate simplices. Note that a cleavage is the same as a pre-cleavage where $C_{n,k}=C_{n,k'}$ for every $k,k'$. 

\begin{proposition}\cite{dht}
Every higher vector bundle $q:V\to G$ admits a normal pre-cleavage $C_{n,k}$. And if it admits a normal cleavage $C_n$, then it also admits a normal weakly flat one.
\end{proposition}

It turns out that a normal pre-cleavage is enough to induce a representation up to homotopy on its core, though the formulas are more involved, and demand subtle simplicial decompositions. Using this as a homotopy inverse for the semi-direct product, we get the following.

\begin{proposition}[\cite{dht}]
The semi-direct product sets a 1-1 correspondence between homotopy classes of representations up to homotopy and of higher vector bundles over $G$.
\end{proposition}

We expect to set an equivalence $\ltimes: Rep^\infty[G]\to VB^\infty[G]$ between the derived categories. The derived category of representations up to homotopy $Rep^\infty[G]$ can be described either as the localization of $Rep^\infty(G)$ by the quasi-isomorphisms or as the quotient of $Rep^\infty(G)$ modulo homotopies \cite[Def. 3.9, Prop. 3.28]{ac}. The construction of $VB^\infty[G]$ should be similar.


\subsection*{Tensor products}

Tensor products of representations up to homotopy arise naturally, but their construction is not canonical, there are choices to be made, and the combinatorics ruling it is rather subtle \cite{acd}. Our intrinsic approach can be of use to get a simpler description of the theory.

The Dold-Kan correspondence admits a monoidal version, built over the Eilenberg-Zilber Theorem, which sets a homotopy equivalence between the tensor product of chain complexes and the tensor product of the corresponding simplicial abelian groups, see e.g. \cite[\S 29]{may}. It is natural to expect a version of this result in our context, where a natural candidate for the tensor product of two higher vector bundles $V,V'\to G$ is defined level-wise:
$$(V\otimes V')_n=V_n\otimes V'_n \qquad d_i(v\otimes v')=d_i(v)\otimes d_i(v') \qquad u_j(v\otimes v')=u_j(v)\otimes u_j(v').$$

\begin{problem}
Develop the theory of tensor products of higher vector bundles, in a way compatible with the tensor product of chain complexes and of usual representations, and relate it to the tensor product of representations up to homotopy \cite{acd} via $\ltimes$.
\end{problem}


Some potential applications of the theory would be to revisit the Bott's spectral sequence of a Lie groupoid $G$ with an intrinsic description of the symmetric powers of $Ad(G)$ \cite{ac}, and to regard multiplicative forms and tensors as morphisms of higher vector bundles among the exterior powers of the adjoint and the co-adjoint representations $Ad(G),Ad^*(G)$ \cite{bc}.





\subsection*{Associated bundle construction for $\infty$-local systems}

In the set-theoretic framework, our semi-direct product gives a new explicit model for the homotopy colimit of the lax functor $R:G\to Ch_{\geq0}(Vect_\R)$, generalizing in some sense the constructions of \cite[3.2.5]{lurie} and \cite{hm}, which only hold for strict functors from a 1-category, see Remark \ref{rmk:final}.
A natural problem is to try to adapt our construction to other more general contexts, replacing $G$ by any simplicial set, and $Ch_{\geq0}(Vect_\R)$ by any dg-category $\mathcal C$.

A similar framework is that of {\bf $\infty$-local system}, introduced in \cite[2.1]{bs}, where given $M$ a smooth manifold, a higher Riemann-Hilbert correspondence is established by integration over simplices:
$$RH:P_{\Omega^\bullet(M)}\to Loc^\infty_{Ch(Vect)}S(M).$$
This relates graded vector bundles $\oplus_n E_n\to M$ with a flat $\Z$-connection, and $\infty$-local systems over the singular simplicial set $S(M)$ with coefficients in $Ch_{\geq0}(Vect_\R)$. A constructible version of this construction is studied in \cite{su}. A problem posed by Block and Smith is the following:

\begin{problem}
Define an inverse functor $Loc^\infty_{Ch(Vect)}S(M)\to P_{\Omega^\bullet(M)}$ realizing the Riemann-Hilbert correspondence, making use of an associated bundle construction.    
\end{problem}

We believe that our semi-direct product may be adapted to solve this problem. From this perspective, the flat $\Z$-connection should be induced by a cleavage on the associated bundle, which would be an intrinsic object.



\subsection*{Lie theory and Morita invariance of higher vector bundles}

The VB-groupoid approach has played a fundamental role in elucidating the Lie theory and the Morita invariance of 2-term representations up to homotopy \cite{bcdh}, \cite{dho}.
Generalizing this, we expect that our semi-direct product construction will play a role tackling these problems for more general representations up to homotopy.

There is a useful approach to vector bundles in the smooth setting, by characterizing them via the scalar multiplication \cite{gr}. This was used in \cite{bcdh} to characterize VB-groupoids and VB-algebroids, and to describe their Lie theory. Analogously, it is natural to characterize higher vector bundles $q: V\to G$ via scalar multiplication and study their differentiation. The infinitesimal counterpart should be Vaintrob's $A$-modules \cite{vaintrob}, which are the intrinsic object associated with Lie algebroid representations up to homotopy \cite{mehta}. A version of this differentiation functor was constructed by Arias Abad and Shaetz \cite{as}, as well as the integration to the $\infty$-dimensional $\infty$-groupoid $\Pi_\infty(A)$. 

\begin{problem}
Show that higher vector bundles differentiate to produce Vaintrob's $A$-modules, and explore when such an $A$-module integrates to a higher vector bundle.
\end{problem}

The Morita invariance of the derived category $Rep^\infty[G]$ was already conjectured in \cite{ac}. In the 2-term case, this is the main result of \cite{dho}, based on the theory of VB-groupoids. Analogously, one would expect to apply the semi-direct product to solve the following:

\begin{problem}
Show that if $q: G'\to G$ is a hypercover between higher Lie groupoids, then the pullback $q^*:VB^\infty[G]\to VB^\infty[G']$ gives an equivalence between the derived categories.    
\end{problem}

A cohomological version was recently proven in \cite{dhos} using the semi-direct product.






{\small

\noindent M. del Hoyo. Universidade Federal Fluminense, Rua Prof. Marcos Waldemar de Freitas Reis, 24210-201, Niter\'oi, Brazil. {\it Email:} mldelhoyo@id.uff.br

\

\noindent G. Trentinaglia. CMAGDS, Instituto Superior T\'ecnico, Universidade de Lisboa, %
      Av.\ Rovisco Pais, 1049-001 Lisboa, Portugal. {\it Email:} gtrentin@math.tecnico.ulisboa.pt

\end{document}